\documentclass[reqno,a4paper,12pt]{amsart}

\usepackage[utf8]{inputenc}
\usepackage[T1]{fontenc}
\usepackage[english]{babel}
\usepackage[babel]{microtype}
\usepackage[a4paper, margin=1.1in]{geometry}
\usepackage[dvipsnames]{xcolor}

\usepackage{csquotes}
\usepackage{amsmath,amssymb,amsthm,amsrefs}
\usepackage{mathrsfs}
\usepackage{mathtools}
\usepackage{commath}

\usepackage{thmtools}
\usepackage{thm-restate}
\usepackage{bbm}
\usepackage{dsfont}
\usepackage{lmodern}
\usepackage{fixcmex}

\usepackage{cases}
\usepackage{enumitem}
\setlist[enumerate,1]{label=(\roman*)}

\usepackage{centernot}
\usepackage{multirow}
\usepackage{float}

\usepackage[pdfborderstyle={/S/U/W 0},unicode]{hyperref}
\hypersetup{
    colorlinks=true,
    linkcolor=blue!85!black,
    citecolor=blue!85!black,
    urlcolor=blue!85!black,
}

\usepackage{cleveref}

\numberwithin{equation}{section}

\ifdefined\thmcolor
\declaretheoremstyle[
  shaded={bgcolor=\thmcolor,
  }
]{plain}
\else
\fi

\ifdefined\defcolor
\declaretheoremstyle[
  headfont=\normalfont\bfseries,
  bodyfont=\normalfont,
  shaded={bgcolor=\defcolor}
]{noital}
\else
\declaretheoremstyle[
  headfont=\normalfont\bfseries,
  bodyfont=\normalfont,
]{noital}
\fi

\declaretheorem[style=plain,numberwithin=section,name=Theorem]{theorem}
\declaretheorem[style=plain,sibling=theorem,name=Proposition]{proposition}
\declaretheorem[style=plain,sibling=theorem,name=Lemma]{lemma}
\declaretheorem[style=plain,sibling=theorem,name=Corollary]{corollary}
\declaretheorem[style=plain,sibling=theorem,name=Conjecture]{conjecture}
\declaretheorem[style=plain,sibling=theorem,name=Problem]{problem}
\declaretheorem[style=noital,sibling=theorem,name=Observation]{observation}

\newcommand{\indef}[1]{\emph{#1}}
\newcommand{\defined}{\mathrel{\coloneqq}}
\newcommand{\defines}{\mathrel{\eqqcolon}}
\DeclarePairedDelimiter{\p}{\lparen}{\rparen}

\renewcommand{\leq}{\leqslant}

\renewcommand{\geq}{\geqslant}
\newcommand{\st}{\mathbin{\colon}}
\undef{\set}
\DeclarePairedDelimiter{\set}{\lbrace}{\rbrace}
\undef{\emptyset}
\newcommand{\emptyset}{\varnothing}
\DeclarePairedDelimiter{\card}{\lvert}{\rvert}
\DeclareMathOperator{\ind}{\mathbf{1}}
\newcommand{\from}{\colon}
\DeclarePairedDelimiter{\floor}{\lfloor}{\rfloor}

\newcommand{\boundary}{\partial}
\undef{\abs}
\DeclarePairedDelimiterX{\abs}[1]
  {\lvert}{\rvert}{\ifblank{#1}{\,\cdot\,}{#1}}
\DeclareMathDelimiter{\given}
  {\mathbin}{symbols}{"6A}{largesymbols}{"0C}
\DeclareMathOperator{\Prob}{\mathbb{P}}
\DeclarePairedDelimiterXPP{\prob}[1]
  {\Prob}{\lparen}{\rparen}{}
  {\renewcommand{\given}{\nonscript\;\delimsize\vert\nonscript\;\mathopen{}}#1}
\DeclareMathOperator{\Expec}{\mathbb{E}}
\DeclarePairedDelimiterXPP{\expec}[1]
  {\Expec}{\lparen}{\rparen}{}
  {\renewcommand{\given}{\nonscript\;\delimsize\vert\nonscript\;\mathopen{}}#1}

\newcommand{\eps}{\varepsilon}
\newcommand{\CC}{\mathbb{C}}
\newcommand{\HH}{\mathbb{H}}

\newcommand{\RR}{\mathbb{R}}
\newcommand{\TT}{\mathbb{T}}
\newcommand{\ZZ}{\mathbb{Z}}
\newcommand{\cC}{\mathcal{C}}
\newcommand{\cE}{\mathcal{E}}
\newcommand{\cG}{\mathcal{G}}
\newcommand{\cH}{\mathcal{H}}
\newcommand{\cP}{\mathcal{P}}
\newcommand{\rB}{\mathrm{B}}
\newcommand{\rM}{\mathrm{M}}
\newcommand{\rR}{\mathrm{R}}
\newcommand{\rT}{\mathrm{T}}
\newcommand{\rU}{\mathrm{U}}

\usepackage{tikz}
\usepackage{algorithm}
\usepackage[noend]{algpseudocode}

\usepackage{pgfplots}
\pgfplotsset{compat=1.15,width=10cm,height=10cm}

\newcommand{\bond}{}
\newcommand{\site}{\operatorname{site}}

\newcommand{\pcbond}{p_{\operatorname{c}}^{\bond}}

\newcommand{\pcmb}[1]{p_{#1}^{\bond}}

\newcommand{\Pbond}[1]{\Prob_{#1}^{\bond}}
\newcommand{\Psite}[1]{\Prob_{#1}^{\,\site}}

\DeclarePairedDelimiterXPP{\pbond}[2]
  {\Pbond{#1}}{\lparen}{\rparen}{}
  {\renewcommand{\given}{\nonscript\;\delimsize\vert\nonscript\;\mathopen{}}
  \ifblank{#2}{\,\cdot\,}{#2}}

\DeclarePairedDelimiterXPP{\psite}[2]
  {\Psite{#1}}{\lparen}{\rparen}{}
  {\renewcommand{\given}{\nonscript\;\delimsize\vert\nonscript\;\mathopen{}}
  \ifblank{#2}{\,\cdot\,}{#2}}
  
\newcommand{\vboundary}{\boundary_{\operatorname{v}}}

\begin{document}

\title[The Maker-Breaker percolation game on a random board]{The Maker-Breaker percolation game \\ on a random board}

\author{Vojt\v{e}ch Dvo\v{r}\'ak \and Adva Mond \and Victor Souza}
\address{Department of Pure Mathematics and Mathematical Statistics (DPMMS), University of Cambridge, Wilberforce Road, Cambridge, CB3 0WA, United Kingdom}
\email{\{vd273,am2759,vss28\}@cam.ac.uk}


\begin{abstract}
The $(m,b)$ Maker-Breaker percolation game on $(\mathbb{Z}^2)_p$, introduced by Day and Falgas-Ravry, is played in the following way.
Before the game starts, each edge of $\mathbb{Z}^2$ is removed independently with probability $1-p$.
After that, Maker chooses a vertex $v_0$ to protect.
Then, in each round Maker and Breaker claim respectively $m$ and $b$ unclaimed edges of $G$.
Breaker wins if after the removal of the edges claimed by him the component of $v_0$ becomes finite, and Maker wins if she can indefinitely prevent Breaker from winning.

We show that for any $p < 1$, Breaker almost surely has a wining strategy for the $(1,1)$ game on $(\mathbb{Z}^2)_p$.
This fully answers a question of Day and Falgas-Ravry, who showed that for $p = 1$ Maker has a winning strategy for the $(1,1)$ game.
Further, we show that in the $(2,1)$ game on $(\mathbb{Z}^2)_p$ Maker almost surely has a winning strategy whenever $p > 0.9402$, while Breaker almost surely has a winning strategy whenever $p < 0.5278$.
This shows that the threshold value of $p$ above which Maker has a winning strategy for the $(2,1)$ game on $\mathbb{Z}^2$ is non-trivial.
In fact, we prove similar results in various settings, including other lattices and biases $(m,b)$.

These results extend also to the most general case, which we introduce, where each edge is given to Maker with probability $\alpha$ and to Breaker with probability $\beta$ before the game starts.
\end{abstract}

\maketitle


\section{Introduction}
\label{sec:intro}

Let $\Lambda$ be an infinite connected graph and let $m$ and $b$ be positive integers.
The \indef{$(m,b)$ Maker-Breaker percolation game on $\Lambda$} is the game played in the following way.
Before the game starts, Maker chooses a vertex $v_0$\footnote{This can be seen as the round zero of the game. In the original formulation by Day and Falgas-Ravry, the vertex $v_0$ was part of the board, not chosen by Maker. This minor modification was made to facilitate the discussion of random boards.}.
After that, Maker and Breaker take turns by alternately claiming unclaimed edges from $\Lambda$.
Maker goes first and claims $m$ unclaimed edges.
In his turn, Breaker claims $b$ unclaimed edges.
Maker's goal is to `protect' $v_0$, meaning to keep it contained in an infinite connected component throughout the game.
If at any given turn, the removal of the edges claimed by Breaker would make the component of $v_0$ finite in $\Lambda$, Breaker wins.
Maker wins if she has a strategy that guarantees that Breaker does not win.
We refer to $\Lambda$ as the \indef{board} of the game and simply write \indef{$(m,b)$ game on $\Lambda$} instead of $(m,b)$ Maker-Breaker percolation game on $\Lambda$.

This game belongs to the broad class of \indef{Maker-Breaker games}, a widely studied type of \indef{positional game}.
The study of positional games has grown into an important area of modern combinatorics, with connections to Ramsey theory, extremal graph theory, probabilistic combinatorics, and to computer science.
The foundational books of Beck \cite{Beck2008-ua} and of Hefetz, Krivelevich, Stojaković and Szabó~\cite{Hefetz2014-qc} elaborates many such connections.
Whenever $(m,b) \neq (1,1)$ we have an example of a \indef{biased Maker-Breaker game}, as originally introduced by Chvátal and Erd\H{o}s \cite{Chvatal1978-hm}.
In the \indef{unbiased} case $(m,b) = (1,1)$, the Maker-Breaker percolation game is a generalisation of the famous Shannon switching game to an infinite board, see Minsky~\cite{Minsky1961-bq} for a full description of the latter.

The Maker-Breaker percolation game, motivated by its analogy with percolation models in statistical physics, was introduced by Day and Falgas-Ravry~\cites{Day2021-sa, Day2021-ua}.
Naturally, the most interesting boards to consider are lattices and in particular $\ZZ^d$.
Among many other things, they have shown that Maker has a simple winning strategy for the $(1,1)$ game on $\ZZ^2$.
Indeed, Maker pairs the edges $(x,y)(x+1,y)$ and $(x,y)(x,y+1)$ so if Breaker claims one of them Maker immediately claims the other.
This guarantees that from any vertex there is always an infinite up-right path which contains no edges claimed by Breaker, and thus Maker wins.
Note that this particular pairing strategy is sensitive to perturbations of the board such as the removal of a few edges.
With that in mind, Day and Falgas-Ravry asked whether Maker can still win on $\ZZ^2$ after the removal of a certain proportion of its edges.

More precisely, one can study the $(m,b)$ game on $(\Lambda)_p$, the random board obtained from $\Lambda$ after bond percolation with parameter $p$ is performed\footnote{This was previously referred to in the literature as the \indef{$p$-polluted} game on $\Lambda$, the main difference being that now $v_0$ can be chosen by Maker. This distinction was necessary as the $(m,b)$ game was previously on a board $(\Lambda, v_0)$, that is, $v_0$ was fixed. We opted to always give Maker the choice of $v_0$.}.
Namely, every edge of $\Lambda$ is independently kept with probability $p$ and removed with probability $1 - p$.
Write $(\Lambda)_p$ for the resulting random subgraph of $\Lambda$, see \Cref{sec:percolation} for detailed definitions.
Denote by $\pcbond(\Lambda)$ the bond percolation threshold for $\Lambda$, that is, the supremum of $p$ such that $(\Lambda)_p$ has no infinite connected component almost surely\footnote{An that occurs with probability $1$ is said to occur \indef{almost surely}.}.

It is a classical result that $\pcbond(\ZZ^2) = 1/2$.
Indeed, Harris~\cite{Harris1960-di} showed that there are no infinite connected components in $(\ZZ^2)_p$ for $p \leq 1/2$, while the celebrated result of Kesten~\cite{Kesten1980-fj} says that there is an infinite connected component if $1/2 < p \leq 1$.
Therefore, Maker has no chance of winning any game on $(\ZZ^2)_p$ with $p \leq 1/2$.
As Maker wins the $(1,1)$ game if $p = 1$, Day and Falgas-Ravry asked for the critical value $p^\ast$ such that Breaker can win the $(1,1)$ game if $p < p^\ast$ and Maker can win if $p > p^\ast$.
Monotonicity clearly implies the existence of $p^\ast$ and in our previous work~\cite{Dvorak2021-ad} we have shown that $p^\ast > 0.6298$.
More precisely, we showed that $p^\ast$ is at least the critical probability for the north-east percolation on $\ZZ^2$.
We now fully answer the question of Day and Falgas-Ravry.

\begin{theorem}
\label{thm:11trivial}
Almost surely Breaker has a winning strategy for the $(1,1)$ Maker-Breaker percolation game on $(\ZZ^2)_p$, for any $0 \leq p < 1$.
\end{theorem}

The strategy of Breaker is to build a pairing on $(\ZZ^2)_p$ that guarantees that Maker cannot leave a large box.
We use bootstrap percolation to build such pairing in a dynamic fashion, which outperforms the static pairing employed in~\cite{Dvorak2021-ad}.

While Maker can indeed win the $(1,1)$ game on $\ZZ^2$, \Cref{thm:11trivial} shows it barely does so.
Whenever a positive proportion of the edges is randomly assigned to Breaker, she can no longer win.
In fact, the same proof extends to the $(1,d-1)$ game on $(\ZZ^d)_p$.

\begin{theorem}
\label{thm:1dtrivial}
Almost surely Breaker has a winning strategy for the $(1,d-1)$ Maker-Breaker percolation game on $(\ZZ^d)_p$, for any $0 \leq p < 1$ and $d \geq 2$.
\end{theorem}

As we will see in \Cref{sec:heuristics}, remarkably little is known for $(m,b)$ games in $\ZZ^d$, $d \geq 3$.

For any integers $m,b \geq 1$ and an infinite connected graph $\Lambda$ we can define the threshold
\begin{equation*}
    \pcmb{(m,b)}(\Lambda)
    \defined \sup\set[\bigg]{ 0 \leq p \leq 1 \st \begin{array}{c}
        \text{almost surely Breaker has a winning } \\
        \text{strategy for the $(m,b)$ game on $(\Lambda)_p$ }
      \end{array} }.
\end{equation*}
The values of the thresholds $\pcmb{(m,b)}(\Lambda)$ can then be interpreted as different measures of connectivity of $\Lambda$ in an adversarial setting.
It is clear, for instance, that $\pcbond(\Lambda) \leq \pcmb{(m,b)}(\Lambda) \leq 1$ for every pair $(m,b)$.
Moreover, the following monotonicity relations hold
\begin{align*}
    \pcmb{(m,b)}(\Lambda) &\leq \pcmb{(m,b+1)}(\Lambda), \\
    \pcmb{(m+1,b)}(\Lambda) &\leq \pcmb{(m,b)}(\Lambda).
\end{align*}

\Cref{thm:1dtrivial} states that $\pcmb{(1,d-1)}(\ZZ^d) = 1$ for $d \geq 2$, so these thresholds are \indef{trivial}.
However, the behaviour at the critical point is not known for $d \geq 3$.
Indeed, whether or not Maker wins the $(1,d - 1)$ game on $\ZZ^d$ is still an open problem, see \Cref{sec:heuristics}.

Turning back to two dimensions, \Cref{thm:11trivial} in a nutshell means that Maker is not strong enough in the $(1,1)$ game on $\ZZ^2$ to win in a robust way.
This prompted us to give more power to Maker and look what happens for the $(2,1)$ game instead.
We prove the following result.

\begin{theorem}
\label{thm:21nontrivial}
Consider the $(2,1)$ Maker-Breaker percolation game on $(\ZZ^2)_p$.
If $p < 0.52784$ then almost surely Breaker has a winning strategy, and if $p > 0.94013$ then almost surely Maker has a winning strategy.
In other words,
\begin{equation*}
    0.52784 \leq \pcmb{(2,1)}(\ZZ^2) \leq 0.94013.
\end{equation*}
\end{theorem}

In particular, the threshold for the $(2,1)$ game in $\ZZ^2$ is \indef{non-trivial}.
To prove this result, we analyse \indef{potential-based} strategies, in which one of the players attempts to minimise some potential function defined on the space of game configurations.

In order to precisely state our next result, recall the following definition.
Given an infinite vertex-transitive graph $\Lambda$, let $\cP^n$ be the set of self-avoiding walks of length $n$ starting from a fixed vertex.
Since $\card{\cP^{n+m}} \leq \card{\cP^n} \card{\cP^m}$, we know that $\card{\cP^n}^{1/n}$ converges to a constant $\kappa = \kappa(\Lambda)$, known as the \indef{connective constant} of $\Lambda$.
The following result gives us a plethora of situations in which Breaker has a winning strategy, even when we only have an upper bound on $\kappa(\Lambda)$.

\begin{theorem}
\label{thm:potential-breaker}
Let $\Lambda$ be an infinite vertex-transitive graph and $\kappa$ be its connective constant.
Then almost surely Breaker has a winning strategy for the $(m,b)$ Maker-Breaker percolation game on $(\Lambda)_p$ whenever
\begin{equation}
\label{eq:mb-breaker-p}
  p < \frac{(b+1)^{1/m}}{\kappa}.
\end{equation}
\end{theorem}

The core idea of \Cref{thm:potential-breaker} is that we provide Breaker with an explicit potential function, called \indef{danger}, that measures how unfavourable a given game configuration is.
Breaker would then attempt to win by keeping the danger as low as possible throughout the game.
The difficulty in this endeavour is twofold: the definition of danger is only useful if it captures enough information from a game configuration, and on the other hand, the danger must be simple enough to be analysed.
The specific forms of the danger functions we use are described in \Cref{sec:potential}.

To prove \Cref{thm:potential-breaker}, we define an appropriate danger function and show that Breaker can win by claiming edges that minimise the danger as long as it is low enough when the game starts.
Finally, we show that the initial danger can be taken to be arbitrarily low whenever condition \eqref{eq:mb-breaker-p} holds.

We emphasise that potential-based strategies may also be used when the board is not random.
Indeed, \Cref{thm:potential-breaker} implies the following result.

\begin{corollary}
\label{cor:deterministic-breaker}
Let $\Lambda$ be an infinite vertex-transitive graph and $\kappa$ be its connective constant.
Then almost surely Breaker has a winning strategy for the $(m,b)$ Maker-Breaker percolation game on $\Lambda$ whenever $b > \kappa^m - 1$.
As a consequence,
\begin{enumerate}
    \item Breaker wins the $(1,1)$ and the $(2,3)$ game on the hexagonal lattice $\HH$;
    \item Breaker wins the $(1,4)$ and the $(2,18)$ game on the triangular lattice $\TT$;
    \item Breaker wins the $(1,2d-2)$ game on $\ZZ^d$ for $d \geq 2$.
\end{enumerate}
\end{corollary}

To get the specific consequences above, we used the known value of $\sqrt{2 + \sqrt{2}}$ for the connective constant of the hexagonal lattice, as famously determined by Duminil-Copin and Smirnov~\cite{Duminil-Copin2012-jj}; the upper bound $\kappa(\TT) \leq 4.278$ from Alm~\cite{Alm1993-tn}; and the fact\footnote{For the strict inequality, see Fisher and Sykes~\cite{Fisher1959-nt}*{Appendix A}.
It is also known that $\kappa(\ZZ^d) = 2d - 1 - 1/d + O(1/d^2)$.
In fact, much more precise asymptotics are available.
See the book of Madras and Slade~\cite{Madras2013-lc} for much more on self-avoiding random walks.} that $\kappa(\ZZ^d) < 2d - 1$.
While Day and Falgas-Ravry~\cite{Day2021-ua} have already determined that Breaker wins the $(1,2)$ game on $\ZZ^2$, \Cref{cor:deterministic-breaker} gives an alternative winning strategy.

We remark that Breaker still has a winning strategy under the conditions of \Cref{thm:potential-breaker} even when Maker can claim an arbitrarily large number of edges in the first turn.
This so-called \indef{boosted} game was introduced to avoid a quick win of Breaker near the origin.
As a consequence, Breaker also wins the boosted version of the games mentioned in \Cref{cor:deterministic-breaker}.

Erd\H{o}s and Selfridge~\cite{Erdos1973-gg} were the first to use potential functions to analyse Maker-Breaker games.
Following their seminal contribution, potential-based strategies have become one of the main tools in the study of positional games.
Their analysis was later extended by Beck~\cite{Beck1982-rh} to biased games, who provided a potential-based strategy for a general Maker-Breaker game.
In the analysis of our potential-based strategies, we are essentially using the result of Beck's.
However, we had to extend the range of applicability to cover infinite boards and boosted games.
Our simple extension of Beck's result is stated formally in \Cref{sec:potential}.

The lower bound on \Cref{thm:21nontrivial} is obtained with the same potential-based strategy.
However, we need some extra work to protect the origin from very short dual cycles.

\begin{theorem}
\label{thm:potential-maker}
Let $\kappa$ be the connective constant of $\ZZ^2$.
Then almost surely Maker has a winning strategy for the $(m,1)$ Maker-Breaker percolation game on $(\ZZ^2)_p$ whenever
\begin{equation}
\label{eq:m1maker-p}
  \frac{m+1}{m}\p[\Big]{1 - \frac{1}{\kappa}} < p \leq 1.
\end{equation}
\end{theorem}

Providing a winning strategy for Maker has proven to be a delicate enterprise.
One reason for this distinction may be that Breaker can win a percolation game in finite time, while Maker only wins by indefinitely avoiding a victory of Breaker.
In fact, to prove \Cref{thm:potential-maker} we need the aforementioned extension of Beck's result to infinite boards.
For many technical reasons, the range of applicability of \Cref{thm:potential-maker} is considerably reduced compared to \Cref{thm:potential-breaker}.
Nonetheless, the proof of \Cref{thm:potential-maker} can be adapted, in principle, to other planar lattices and to pairs $(m,b)$ with $b > 1$.

We note that \Cref{thm:potential-maker} can also be used to show that Maker wins on a deterministic board.
For example, having $\kappa(\ZZ^2) < 3$ implies that Maker wins the $(2,1)$ game on $\ZZ^2$.
While this has been previously determined by Day and Falgas-Ravry~\cite{Day2021-ua}, \Cref{thm:potential-maker} gives an alternative winning strategy for Maker.

\subsection{Phase diagrams for percolation games}
\label{subsec:phase-diagrams}

Consider a Maker-Breaker game where the board is the edge set of the complete graph $K_n$ and the winning sets are a collection of subgraphs of $K_n$.
This is the case for many well studied games in the literature, such as the connectivity game, the Hamiltonicity game and the $H$ game\footnote{These are the Maker-Breaker games played on the edge set of $K_n$ where the winning sets are, respectively, all the spanning trees, all the Hamiltonian cycles and all copies of $H$ in $K_n$.}.
For many of these games, it turns out that Maker wins rather easily when $(m,b) = (1,1)$.
A natural question arises: can Maker still win if the game is modified in way to make her objective slightly harder?

There are several different ways to alter a Maker-Breaker game in order to decrease the advantage of Maker.
One way is to provide Breaker with more moves in each turn, as in biased $(1,b)$ games.
This direction, originally proposed by Chvátal and Erd\H{o}s~\cite{Chvatal1978-hm}, invites the natural question of determining the value of the \indef{threshold bias}, that is, the integer $b^\ast$ such that Maker wins the $(1,b)$ game if $b \leq b^\ast$ and Breaker wins the $(1,b)$ game if $b > b^\ast$.
The existence of $b^*$ is guaranteed by bias monotonicity and the fact that Breaker wins if $b = \binom{n}{2}$.

Another way to decrease Maker's advantage is randomly assigning board elements to Breaker at the beginning of the game.
In this direction, Stojaković and Szabó~\cite{Stojakovic2005-bu} considered games played on random boards.
Perhaps the most notable example of a random board is the binomial random graph $G(n,p)$, in which one can think of the missing edges as given to Breaker.
Now the natural parameter to consider is the threshold $p^\ast$ such that with high probability Breaker has a winning strategy for the $(1,1)$ game on $G(n,p)$ if $p = o(p^\ast)$, and with high probability Maker wins the $(1,1)$ game on $G(n,p)$ if $p = \omega(p^\ast)$.
The existence of $p^\ast$ is guaranteed by the threshold theorem of Bollobás and Thomason~\cite{Bollobas1987-fk}, as the property `Maker has a winning strategy for the $(1,1)$ game on $G(n,p)$' is monotone and non-trivial.

In the study of positional games the probabilistic viewpoint has proven to be useful.
For certain games, a strategy where Maker claims board elements randomly has shown to be essentially optimal.
In a $(1,b)$ game a $1/(b+1)$ proportion of the edges are claimed by Maker, so by playing randomly, the graph spanned by Maker's edges at the end of the game is similar to $G(n,1/(b+1))$.
If this strategy is optimal, this suggest that $1/(b^\ast+1)$ is of the same order of magnitude as the threshold of the appearance of a winning set in $G(n,p)$.
This \indef{probabilistic intuition} has been rigorously established for many games.
For instance, Gebauer and Szabó~\cites{Gebauer2009-ha} have shown that it holds for the connectivity game, while Krivelevich~\cite{Krivelevich2011-rv} established the same for the Hamiltonicity game.
However, this prediction fails for other games, such as the $H$ game for a fixed graph $H$, as shown by Bednarska and {\L}uczak~\cite{Bednarska2000-gs}.

As previously discussed, even though Maker wins the $(1,1)$ Maker-Breaker percolation game on $\ZZ^2$, her strategy is quite fragile.
Indeed, we know that Breaker wins the $(1,2)$ game on $\ZZ^2$ and that, from \Cref{thm:11trivial}, Breaker wins the $(1,1)$ game on $(\ZZ^2)_p$ for all $0 \leq p < 1$.
In this scenario, it also makes sense to consider ways to increase the power of Maker.
In fact, we have already done so in \Cref{thm:21nontrivial} where we address the $(2,1)$ game on $(\ZZ^2)_p$.
On the other hand, randomly assigning edges to Maker in the beginning of the game has not been considered in the literature before.

For $\alpha, \beta \geq 0$, $\alpha + \beta \leq 1$, let $(\Lambda)_{\alpha,\beta}$ be the random configuration where every edge of $\Lambda$ is independently given to Maker with probability $\alpha$, to Breaker with probability $\beta$, and remains unclaimed with probability $1 - \alpha - \beta$.
When we refer to the $(m,b)$ game on $(\Lambda)_{\alpha,\beta}$ we are considering the Maker-Breaker percolation game on $\Lambda$ where the initial configuration is drawn from the distribution $(\Lambda)_{\alpha,\beta}$.
As before, Maker can choose $v_0$ after the initial configuration is determined.
See \Cref{sec:percolation} for detailed definitions.

Let us consider the simplest case of the $(1,1)$ game on $\ZZ^2$.
As Maker wins on $(\ZZ^2)_{0,0}$, by monotonicity, she wins almost surely on $(\ZZ^2)_{\alpha,0}$ for all $0 \leq \alpha \leq 1$.
By \Cref{thm:11trivial} we know that almost surely Breaker wins on $(\ZZ^2)_{0,\beta}$ for all $0 < \beta \leq 1$.
If $\alpha + \beta = 1$ then there are no edges left unclaimed and the game ends on round zero.
Indeed, on $(\ZZ^2)_{p,1-p}$, almost surely Maker wins if $1/2 < p \leq 1$ and almost surely Breaker wins if $0 \leq p \leq 1/2$.
In fact, this is equivalent to the theorems of Harris and Kesten.

\begin{figure}[ht!]
\centering
\begin{tikzpicture}[thick, scale=1.0]
\begin{axis}[
    width=10em,height=10em,
    scale only axis, axis equal,
    axis x line* = bottom,
    axis y line* = left,
    axis line style={draw=none},
    ymin=-0.04, ymax=1.04,
    xmin=-0.04, xmax=1.04,
    xtick = {0, 1/2, 1},
    xticklabels = {0, 1/2, 1},
    ytick = {0, 1/2, 1},
    yticklabels = {0, 1/2, 1},
    xlabel=$\alpha$,
    ylabel=$\;\;\beta$,
    every axis x label/.style={at={(current axis.right of origin)},anchor=west},
    every axis y label/.style={at={(current axis.north west)},anchor=south},
]
\path[fill=red!20] (0,0.5) -- (0,0.2535) -- (0.24649,0.5) -- (0.24649,0.6) -- (0,0.6) -- cycle;
\path[fill=blue!20] (0.5,0) -- (0.2535,0) -- (0.5,0.24649) -- (0.6,0.24649) -- (0.6, 0) -- cycle;
\draw [ultra thick, blue] (0,0) -- (1,0);
\draw [ultra thick, red] (0,0) -- (0,1);
\filldraw[ultra thick, draw=red, fill=red!20] (0,0.5) -- (0,1) -- (0.5,0.5) -- (0.24649,0.5);
\filldraw[ultra thick, draw=blue, fill=blue!20] (0.5,0) -- (1,0) -- (0.5,0.5);
\filldraw[blue] (0,0) circle (0.2em);
\filldraw[red] (0.5,0.5) circle (0.2em);
\end{axis}
\end{tikzpicture}
\caption{The known regions of the phase diagram for the $(1,1)$ game on $(\ZZ^2)_{\alpha,\beta}$. Almost surely, Breaker wins for points coloured red and Maker wins for points coloured blue.}
\label{fig:phase11}
\end{figure}
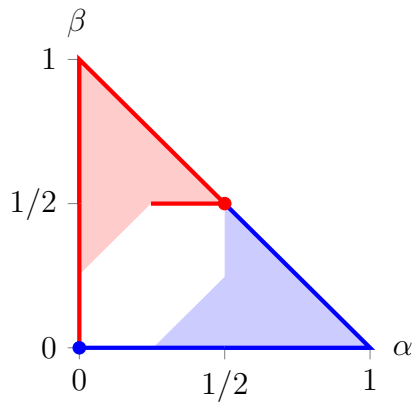

So far we have determined who wins the $(1,1)$ game on $(\ZZ^2)_{\alpha,\beta}$ only at the boundary of the triangle $\rT \defined \set{\alpha, \beta \geq 0 \st \alpha + \beta \leq 1}$.
The event `Maker has a winning strategy for the $(m,b)$ game' is translation invariant as Maker can choose $v_0$.
Therefore, Kolmogorov's zero--one law implies that either Maker or Breaker has a winning strategy almost surely.
Thus, there is a partition $\rT = \rM \sqcup \rB$ such that Maker wins the $(1,1)$ game on $(\ZZ^2)_{\alpha,\beta}$ if and only if $(\alpha,\beta) \in \rM$.
We represent this partition graphically in a \indef{phase diagram} by colouring points in $\rM$ in blue and points in $\rB$ in red.
See for instance \Cref{fig:phase11} for the know regions of the phase diagram for the $(1,1)$ game on $(\ZZ^2)_{\alpha,\beta}$.

In a phase diagram, we have the following monotonicity relations:
\begin{alignat*}{2}
    (\alpha, \beta) \in \rM &\implies (\alpha', \beta) \in \rM && \qquad \text{ whenever $\alpha' \geq \alpha$ and $(\alpha', \beta) \in \rT$}, \\
    (\alpha, \beta) \in \rB &\implies (\alpha, \beta') \in \rB && \qquad \text{ whenever $\beta' \geq \beta$ and $(\alpha, \beta') \in \rT$}.
\end{alignat*}
Thus, if we define $\varphi \from [0,1] \to [0,1]$ as $\varphi(\alpha) \defined \sup \set[\big]{ \beta \geq 0 \st (\alpha, \beta) \in \rM}$, then $\varphi$ is monotone non-decreasing.
While $\varphi$ is may not be continuous, it can only have jump discontinuities.
Therefore, the graph of $\varphi$ extends to a continuous curve in $\rT$ that separates $\rM$ and $\rB$ called the \indef{phase boundary}.

For the $(1,1)$ game on $(\ZZ^2)_{\alpha,\beta}$ the monotonicity relations above imply that the triangle $\set{ (\alpha, \beta) \in \rT \st \alpha > 1/2}$ is contained in $\rM$ and that the triangle $\set{ (\alpha, \beta) \in \rT \st \beta \geq 1/2}$ is contained in $\rB$.
While we are unable to determine the separating curve for this game, we are able to extend the aforementioned triangles to larger subregions of $\rM$ and $\rB$.
We obtain the regions of the phase diagram in \Cref{fig:phase11} via the following two theorems.

\begin{theorem}
\label{thm:potential-breaker-symmetric}
Let $\Lambda$ be an infinite vertex-transitive graph with connectivity constant bounded above by $\kappa$.
Then almost surely Breaker has a winning strategy for the boosted $(m,b)$ Maker-Breaker percolation game on $(\Lambda)_{\alpha,\beta}$ whenever
\begin{equation}
\label{eq:mb-breaker-ab}
  (1 - \alpha - \beta) < (b+1)^{1/m}\p[\Big]{\frac{1}{\kappa} - \alpha}.
\end{equation}
\end{theorem}

Note that \Cref{thm:potential-breaker-symmetric} extends \Cref{thm:potential-breaker} as the latter can be obtained by taking $\alpha = 0$ and no boost.
Similarly, the next theorem extends \Cref{thm:potential-maker}. 

\begin{theorem}
\label{thm:potential-maker-symmetric}
Let $\kappa$ be the connective constant of $\ZZ^2$, let $\alpha, \beta \geq 0$, $\alpha + \beta \leq 1$, and $m \geq 2$.
Then almost surely Maker has a winning strategy for the $(m,1)$ Maker-Breaker percolation game on $(\ZZ^2)_{\alpha,\beta}$ whenever
\begin{equation}
\label{eq:mb-maker-ab}
    (1 - \alpha - \beta) < (m+1)\p[\Big]{\frac{1}{\kappa} - \beta}.
\end{equation}
\end{theorem}

\Cref{thm:potential-breaker-symmetric} and \Cref{thm:potential-maker-symmetric} allow us to obtain non-trivial information on the phase diagrams of $(m,b)$ games on $(\ZZ^2)_{\alpha,\beta}$ for other pairs $(m,b)$.
See \Cref{fig:phase2112} below.

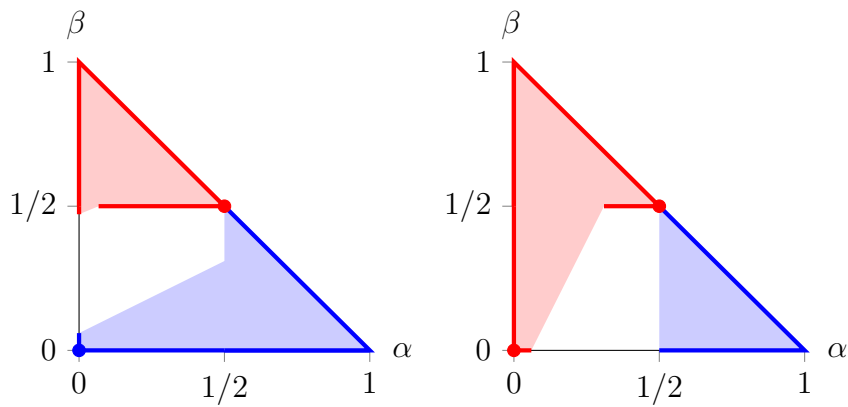
\begin{figure}[ht!]
\centering
\begin{tikzpicture}[thick, scale=1]
\begin{axis}[
    width=10em,height=10em,
    scale only axis, axis equal,
    axis x line* = bottom,
    axis y line* = left,
    axis line style={draw=none},
    ymin=-0.04, ymax=1.04,
    xmin=-0.04, xmax=1.04,
    xtick = {0, 1/2, 1},
    xticklabels = {0, 1/2, 1},
    ytick = {0, 1/2, 1},
    yticklabels = {0, 1/2, 1},
    xlabel=$\alpha$,
    ylabel=$\;\;\beta$,
    every axis x label/.style={at={(current axis.right of origin)},anchor=west},
    every axis y label/.style={at={(current axis.north west)},anchor=south},
]
\path[fill=red!20] (0,0.5) -- (0,0.47215) -- (0.06723,0.5) -- (0.06723,0.6) -- (0,0.6) -- cycle;
\path[fill=blue!20] (0,0.059868) -- (0.5,0.30986) -- (0.6,0.30986) -- (0.6,0) -- (0, 0) -- cycle;
\draw [black] (0,0) -- (0,1);
\draw [ultra thick, blue] (0,0.059868)-- (0,0) -- (1,0);
\draw [ultra thick, red] (0,0.47215) -- (0,1);
\filldraw[ultra thick, draw=red, fill=red!20] (0,0.5) -- (0,1) -- (0.5,0.5) -- (0.06723,0.5);
\filldraw[ultra thick, draw=blue, fill=blue!20] (0.5,0) -- (1,0) -- (0.5,0.5);
\filldraw[blue] (0,0) circle (0.2em);
\filldraw[red] (0.5,0.5) circle (0.2em);
\end{axis}
\end{tikzpicture}
\begin{tikzpicture}[thick, scale=1]
\begin{axis}[
    width=10em,height=10em,
    scale only axis, axis equal,
    axis x line* = bottom,
    axis y line* = left,
    axis line style={draw=none},
    ymin=-0.04, ymax=1.04,
    xmin=-0.04, xmax=1.04,
    xtick = {0, 1/2, 1},
    xticklabels = {0, 1/2, 1},
    ytick = {0, 1/2, 1},
    yticklabels = {0, 1/2, 1},
    xlabel=$\alpha$,
    ylabel=$\;\;\beta$,
    every axis x label/.style={at={(current axis.right of origin)},anchor=west},
    every axis y label/.style={at={(current axis.north west)},anchor=south},
]
\path[fill=red!20] (0,0)--(0.059868,0)--(0.30986,0.5)--(0.30986,0.6)--(0,0.6)--cycle;
\draw [black] (0,0) -- (1,0);
\draw [ultra thick, red] (0.059868,0) -- (0,0) -- (0,1);
\filldraw[ultra thick, draw=red, fill=red!20] (0,0.5) -- (0,1) -- (0.5,0.5) -- (0.30986,0.5);
\filldraw[ultra thick, draw=blue, fill=blue!20] (0.5,0) -- (1,0) -- (0.5,0.5);
\filldraw[red] (0,0) circle (0.2em);
\filldraw[red] (0.5,0.5) circle (0.2em);
\end{axis}
\end{tikzpicture}
\caption{The known regions of the phase diagram for the $(2,1)$ game on the left, and for the $(1,2)$ game on the right, both played on $(\ZZ^2)_{\alpha,\beta}$.}
\label{fig:phase2112}
\end{figure}

Recall that with the bound $\kappa(\ZZ^2) < 3$ we could deduce from \Cref{thm:potential-breaker} that Breaker wins the $(1,2)$ game on $\ZZ^2$ and from \Cref{thm:potential-maker} that Maker wins the $(2,1)$ game on $\ZZ^2$.
Now, equipped with \Cref{thm:potential-breaker-symmetric} and \Cref{thm:potential-maker-symmetric}, and with the bound $\kappa(\ZZ^2) \leq 2.6792$ determined by Pönitz and Tittmann~\cite{Ponitz2000-yg}, we can extend these results and get the regions as in \Cref{fig:phase2112}.
Indeed, with the improved bound on $\kappa(\ZZ^2)$ we know that, almost surely, Breaker wins the $(1,2)$ game on $(\ZZ^2)_{\alpha,0}$ if $\alpha \leq 0.05987$ and Maker wins the $(2,1)$ game on $(\ZZ^2)_{0,\beta}$ if $\beta \leq 0.05987$.
Moreover, the same upper bound on $\kappa(\ZZ^2)$ was used when determining the regions in \Cref{fig:phase11}.

\subsection{Notation}

For a graph $\Lambda$, we denote by $V(\Lambda)$ its vertex set, and by $E(\Lambda)$ its edge set.
The \indef{vertex boundary} $\vboundary H$ of a finite subgraph $H$ of a graph $\Lambda$ is defined as
\begin{equation*}
    \vboundary H \defined \set[\big]{ x \in V(H) \st  \set{x,y} \in E(\Lambda) \text{ for some $y \in V(\Lambda) \setminus V(H)$} }.
\end{equation*}

The lattice $\ZZ^d$ is infinite graph with the following vertex and edge sets:
\begin{align*}
    V(\ZZ^d) &\defined \set[\big]{ (x_1, \dotsc, x_d) \st x_1, \dotsc, x_d \in \ZZ }, \\
    E(\ZZ^d) &\defined \set[\big]{ \set{(x_1, \dotsc, x_d), (y_1, \dotsc, y_d)} \subseteq\ZZ^d \st
    \abs{x_1 - y_1} + \dotsb + \abs{x_d - y_d} = 1}.
\end{align*}
The \indef{box of radius $m$ centred at $x$}, denoted by $B_m(x)$, is the subgraph of $\ZZ^d$ induced by the vertices $(x + [-m,m]^d) \cap \ZZ^d$.
A path $P = v_1 \dotsb v_\ell$ in a graph is a sequence of distinct neighbouring vertices, and we refer to $v_1$ and $v_\ell$ as the \indef{endpoints} of $P$ and to $v_2, \dotsc, v_{\ell - 1}$ as the \indef{internal vertices} of $P$.

Given a plane graph $\Lambda$, its plane dual is the graph $\Lambda'$ with edge set being the set of faces in the plane realisation of $\Lambda$, and two faces are connected by an edge if they share a boundary edge.
There is a canonical map $E(\Lambda) \to E(\Lambda')$ that sends an edge $e \in E(\Lambda)$ to its dual edge $e'$, that connects the two faces that share $e$.

Given an infinite planar connected graph $\Lambda$ and a vertex $v_0 \in V(\Lambda)$, an separating set is a set of edges $S \subseteq E(\Lambda)$ whose removal makes the component of $v_0$ finite.
Note that the dual of a minimal separating set correspond exactly as a dual cycle that encircles $v_0$ in a bounded region.
For that reason, when we consider a game played on $\ZZ^2$, we often consider Breaker as trying to claim a dual cycle around $v_0$, as he can pretend his moves are made in the dual graph, rather than in $\ZZ^2$.

\subsection{Structure}

Some notations and preliminary facts about percolation and random configurations are established in \Cref{sec:percolation}.
In, \Cref{sec:bootstrap} we use facts from bootstrap percolation to deduce \Cref{thm:1dtrivial}.
We introduce the general Maker-Breaker game on a hypergraph in \Cref{sec:hyper}, where we also describe the hypergraphs that we will use.
In \Cref{sec:potential} we outline the potential strategies for Maker and Breaker, providing a detailed proof of \Cref{thm:potential-breaker-symmetric} and \Cref{thm:potential-maker-symmetric}.
In \Cref{sec:analysis}, we provide our analysis of the potential-based strategy in a general Maker-Breaker game.
Finally, we summarise the state of the art of percolation games in \Cref{sec:heuristics}, listing many open problems and conjectures.


\section{Percolation, random boards and initial configurations}
\label{sec:percolation}

In this section we introduce the basic definitions and recall some standard results of bond percolation, which is the natural model for a \indef{random board} for a Maker-Breaker percolation game.
We mostly follow the notation and terminology from Bollobás and Riordan~\cite{Bollobas2006-jp}.

Consider an infinite connected graph $\Lambda$ and $0 \leq p \leq 1$.
In \indef{bond percolation}, each edge of $\Lambda$ is declared \indef{open} with probability $p$, independently from all other edges.
An edge that is not open is said to be \indef{closed}.

To be precise, a \indef{bond configuration} on $\Lambda$ is a function $\omega \from E(\Lambda) \to \set{0,1}$, mapping an edge $e$ into $\omega_e$.
The space of all bond configurations of $\Lambda$ is $\Omega_\Lambda^{\bond} \defined \set{0,1}^{E(\Lambda)}$.
We then say that $e$ is open if $\omega_e = 1$ and closed if $\omega_e = 0$.
Write $\Pbond{\Lambda, p}$ for the product measure on $\Omega_{\Lambda}^{\bond}$ with $\pbond{\Lambda, p}{w_e = 1} = p$ for every $e \in E(\Lambda)$.
We also write $(\Lambda)_p^{\bond}$ for a random configuration with law $\Pbond{\Lambda, p}$.
An \indef{open cluster} is a connected component that consists of open edges.
The central question in bond percolation is the study of the existence of an infinite open clusters under the law of $\Pbond{\Lambda, p}$ as $p$ varies from $0$ to $1$.

When $\Lambda = \ZZ^2$, the bond percolation model undergoes a phase transition at $p = 1/2$.
This was determined by a combination of the remarkable results of Harris~\cite{Harris1960-di}, who showed that $\pbond{\ZZ^2,p}{\text{there is an infinite open cluster}} = 0$ for all $p \leq 1/2$, and Kesten~\cite{Kesten1980-fj}, who showed that $\pbond{\ZZ^2,p}{\text{there is an infinite open cluster}} = 1$ for all $p > 1/2$.
Simpler proofs of these classical results were later found, for instance, by Bollobás and Riordan~\cites{Bollobas2006-rk,Bollobas2007-zp}.

Correlation inequalities are important tools for the study of percolation theory.
First note that we can induce a partial order on the $\Omega_{\Lambda}^{\bond}$ by defining that $\omega \preceq \omega'$ if $\omega_e = 1$ implies $\omega'_e = 1$ for every edge $e \in E(\Lambda)$.
We say that an event $\cE \subseteq \Omega_{\lambda}^{\bond}$ is increasing if $\omega \in \cE$ and $\omega \preceq \omega'$ implies that $\omega' \in \cE$.
We will make use of the correlation inequality of Harris~\cite{Harris1960-di}.

\begin{lemma}[Harris' inequality]
\label{lem:harris}
Let $\cE_1, \dotsc, \cE_n \subseteq \Omega_{\Lambda}^{\bond}$ be increasing events.
Then under the product measure $\Pbond{\Lambda, p}$, we have
\begin{equation*}
    \pbond{\Lambda,p}{\cE_1 \cap \dotsb \cap \cE_k} \geq \pbond{\Lambda,p}{\cE_1} \dotsm \pbond{\Lambda,p}{\cE_k}.
\end{equation*}
\end{lemma}
While bond percolation is the random boards for the percolation game, site percolation will be used in our analysis of \Cref{thm:1dtrivial}.

In \indef{site percolation}, each vertex of $\Lambda$ is open with probability $p$ and closed otherwise, independently of all other vertices.
The \indef{site configurations} $\Omega_\Lambda^{\site} = \set{0,1}^{V(\Lambda)}$ receive the product measure $\Psite{\Lambda, p}$ with $\psite{\Lambda,p}{\omega_v = 1} = p$ for every $v \in V(\Lambda)$.
We denote by $(\Lambda)_p^{\site}$ a random configuration with law $\Psite{\Lambda, p}$.

\subsection{Random initial configurations}

It will become convenient in \Cref{sec:hyper} and onwards to adopt the point of view of random initial configurations rather than that of percolation.
A \indef{game configuration} is a function $\sigma \from E(\Lambda) \to \set{\rU, \rM, \rB}$, $e \mapsto \sigma_e$.
Given a configuration $\sigma$, we say that a vertex $v \in V(\Lambda)$ is \indef{of Maker} if $\sigma(v) = \rM$, \indef{of Breaker} if $\sigma(v) = \rB$, and \indef{unclaimed} if $\sigma(v) = \rU$.
Write $\cG_\Lambda \defined \set{\rU, \rM, \rB}^{E(\Lambda)}$ for the space of all game configurations.

For $\alpha, \beta \geq 0$, with $\alpha + \beta \leq 1$, write $\Pbond{\Lambda,\alpha,\beta}$ for the product measure on $\cG_\Lambda$ with $\pbond{\Lambda,\alpha,\beta}{\sigma_e = \rM} = \alpha$, $\pbond{\ZZ^2,\alpha,\beta}{\sigma_e = \rB} = \beta$ and $\pbond{\ZZ^2,\alpha,\beta}{\sigma_e = \rU} = 1 - \alpha - \beta$.
We say that a game configuration $\sigma$ is distributed as $(\Lambda)_{\alpha, \beta}$ if it is drawn from the measure $\Pbond{\ZZ^2,\alpha,\beta}$.
In other words, each edge is assigned to Maker or Breaker probabilities $\alpha$ and $\beta$, respectively, and it is left unclaimed otherwise.

The Maker-Breaker percolation game on a random board $(\Lambda)_p$ is equivalent to the Maker-Breaker percolation game with random initial configuration distributed as $(\Lambda)_{0,1-p}$.
This is so as giving an edge to Breaker and removing it altogether from the board has the same effect for the game.
Considering the game with a random initial configuration $\sigma_0$ distributed as $(\Lambda)_{\alpha,\beta}$ has the advantage of considering a bias towards Maker, which cannot be easily expressed in terms of a random board.


\section{The \texorpdfstring{$(1,d-1)$}{(1,d-1)} game on \texorpdfstring{$(\mathbb{Z}^d)_p$}{Zdp}}
\label{sec:bootstrap}

The goal of this section is to prove \Cref{thm:1dtrivial}.
In other words, we want to show that if $p < 1$, then Breaker almost surely has a winning strategy for the $(1,d-1)$ game on $(\ZZ^d)_p$.
Before diving into the details of the proof, we will sketch the case $d = 2$, namely \Cref{thm:11trivial}.

\subsection{Sketch of the \texorpdfstring{$d = 2$}{d = 2} case}

Our goal as Breaker is to eventually fully claim a self-avoiding dual cycle around the origin.
Alternatively, we could guarantee that Maker will never claim a self-avoiding walk from $v_0$ to the boundary of some large box centred at $v_0$ and force all the edges in this box to be eventually claimed.

Consider $\ZZ^2$ after bond percolation with parameter $p < 1$ was performed.
Each vertex has degree $\leq 1$ with positive probability.
These vertices can be safely removed from the board, as they are not present in any self-avoiding dual cycle around $v_0$.
Instead of actually removing these vertices, we paint them red, marking that they are not useful for Maker.
Paint all other vertices blue, as they could still be useful.

Now, if a blue vertex $v$ has $\leq 1$ blue neighbour, they too cannot be used, even if their actual degree is higher.
Indeed, $v$ can only be an internal vertex of a path that uses a red vertex.
So we can paint these vertices red too.
After repeating this process indefinitely, the blue vertices form the $2$-core\footnote{The $k$-core of a graph is obtained by repeatedly removing all vertices with smallest degree until only vertices with degree $\geq k$ remain. The order of removal does not change the final graph. Note that the $k$-core of a graph may be empty, even if there are vertices of degree larger than $k$ to begin with.} of $(\ZZ^2)_p$.

If $p > 1/2$, the $2$-core of $(\ZZ^2)_p$ is infinite, so there is still some hope for Maker.
But Breaker can play in a way to render additional vertices useless.
Indeed, if a blue vertex $v$ has exactly $2$ blue neighbours, say $x$ and $y$, then Breaker can pair the edges $vx$ and $vy$.
In other words, if Maker claims $vx$, Breaker claims $vy$, and vice versa.
Doing so ensures that $v$ cannot be an internal vertex of a useful path of Maker, that is, a path using only blue vertices.
We can now safely paint this vertex red.
This process can be continued indefinitely before the game even starts and the remaining blue vertices form the $3$-core of $(\ZZ^2)_p$.

To recapitulate, we started by setting each vertex red if it has degree $0$ or $1$ and blue otherwise.
If a blue vertex has $0$ or $1$ blue neighbours, it is useless so we paint it red.
If a blue vertex has $2$ blue neighbours, then Breaker builds the pairing of the two edges and we also paint the vertex red.
That is to say, a blue vertex becomes red if it has $2$ or more red neighbours.
The red vertices follow the $2$-neighbour bootstrap percolation.

The classical result of van Enter guarantees that if each vertex of $\ZZ^2$ is infected with probability $q > 0$ independently, the the $2$-neighbour bootstrap percolation process infects every vertex.
While each vertex is in our original set was present with probability $(1-p)^4 + 4(1-p)^3 > 0$, they are not independent.
Instead, we can start the process by infecting the vertices where the north and east edges are missing, possibly pairing their neighbours.
Then each vertex is initially infected independently with probability $(1-p)^2 > 0$ and van Enter results apply: every vertex eventually becomes red.

The last piece in the proof \Cref{thm:11trivial} is a compactness argument.
To be more precise, we allow Breaker to choose a large enough radius for a box around $v_0$ with the property that every open path from $v_0$ to the boundary passes through both edges of the pairing.
Breaker will always play inside this box, ignoring the paired edges outside.
Clearly, preventing Maker from leaving this box is enough for Breaker to guarantee his victory.
Indeed, this way, the game eventually stops and no path from $v_0$ to the boundary of the box is claimed by Maker.

\subsection{A strategy for Breaker}

We can now proceed to a formal proof of \Cref{thm:1dtrivial}.
Recall that $B_m(x)$ is the box of radius $m$ centred at $x$, namely, the subgraph of $\ZZ^d$ induced by the vertices $(x + [-m,m]^d) \cap \ZZ^d$.
We say that a vertex $x \in B_m(v_0)$ is \indef{internal} if $x \neq v_0$ and $x \notin \vboundary B_m(v_0)$.
Let $\omega$ be a bond configuration on $\ZZ^d$.
We say that a vertex $v \in \ZZ^d$ is \indef{$m$-confined in $\omega$} if there is no open subgraph $H$ of the box $B_m(v)$ such that all of the following conditions hold:
\begin{enumerate}
    \item $v \in V(H)$;
    \item $\card{V(H) \cap \vboundary B_m(v)} \geq 1$;
    \item $\deg_H(x) \geq d + 1$ for any every internal vertex $x \in V(H)$.
\end{enumerate}
We will soon see that if $v$ is $m$-confined in $\omega$, then if we repeatedly removed internal vertices of degree $\leq d$ from $B_m(v)$, then we would disconnect $v$ from $\vboundary B_m(v)$.

We say that a configuration $\omega$ is \indef{confined} if every vertex $x \in \ZZ^d$ is $m$-confined in $\omega$ for some $m = m(x) \geq 2$.
In \Cref{prop:confined}, which we prove later, we show that if $p < 1$ then almost surely $(\ZZ^d)_p$ is confined.
Given that, in order to finish the proof of \Cref{thm:1dtrivial} it suffices to show that Breaker can win the $(1,d-1)$ game on a confined configuration of $\ZZ^d$.

\begin{proposition}
\label{prop:confined-strategy}
Breaker has a winning strategy for a $(1,d-1)$ game on any confined configuration of $\ZZ^d$.
\end{proposition}

Consider a confined configuration $\omega$.
Regardless of Maker's choice of $v_0 \in \ZZ^d$, choose a radius $m$ such that $v_0$ is $m$-confined in $\omega$.

We now construct a family $G$ of subsets of open edges of $B_m(v_0)$.
We refer to the subsets in the family $G$ as \indef{clumps}.
Indeed, we do so by following \Cref{alg:core-construction} listed below, which is essentially the construction of the $(d+1)$-core of the open edges in $B_m(v_0)$.
The clumps are a generalisation of the pairing discussed for $d = 2$.
Now each clump has between $2$ and $d$ edges and they are edge-disjoint.
It may be useful to have in mind that Breaker will later play in a way that if Maker ever claims an edge in a clump, he will claim all other edges.
We keep track of the used vertices $U$ and the order in which they appear.

\begin{algorithm}
\caption{Construction of a family of clumps $G$ in the box $B_m(v_0)$}
\label{alg:core-construction}
\begin{algorithmic}[1]
\State $F \gets \text{subgraph of open edges of $B_m(v_0)$}$
\State $G,U \gets \emptyset$
\State $i \gets 1$
\While{$F$ has a path from $v_0$ to $\vboundary B_m(v_0)$}
    \State Let $u_i$ be an internal vertex minimising $\deg_F(u_i)$
    \If{ $\deg_F(u_i) \geq d + 1$} $ $ terminate
    \ElsIf{ $2 \leq \deg_{F}(u_i) \leq d$}
        \State $S_i \gets \set{u_i y_1, \dotsc, u_i y_\ell}$, where $y_1, \dotsc, y_\ell$ are the neighbours of $u_i$ in $F$
        \State $G \gets G \cup \set{S_i}$
    \EndIf
    \State $U \gets U \cup \set{u_i}$
    \State $F \gets F - \set{u_i}$
    \State $i \gets i + 1$
\EndWhile
\end{algorithmic}
\end{algorithm}

First notice that the algorithm always terminates as $F$ has finitely many vertices and each iteration either removes a vertex from $F$, or terminates altogether.
The algorithm constructs a set $U = \set{u_1, u_2, \dotsc}$ of used vertices and a family $G$ of clumps.
Each clump $S_i$ is of the form $\set{u_i y_1, \dotsc, u_i y_\ell}$.
We say that the clump $S_i$ is \indef{centred} at the vertex $u_i$.
Once the clump $S_i$ is added to $G$, the centre $u_i$ is removed from $F$, so the clumps are all edge-disjoint as promised.
A vertex $u_j \in U$ may or may not be the centre of some clump.

\begin{figure}[ht!]
\centering

\begin{tikzpicture}[scale=1.0]
\tikzstyle{s-helper}=[black, line width=0.1em];
\definecolor{s-col}{RGB}{135, 201, 180};
\tikzstyle{s-line}=[line width=0.5em, s-col, draw opacity=1];
\tikzstyle{s-node}=[fill=white,inner sep=0.4pt, circle, line width=0.6pt,draw,font=\scriptsize];

\def\d{0.3}
\def\s{0.08}

\draw[s-line]
(-3,0+\s)--(-3,-1+\d)
(0+\d,-2)--(1,-2)--(1,-1-\d)
(0+\d,-1)--(1,-1)--(1,0-\d)
(3,0+\d)--(3,1)--(2+\d,1)
(2,1+\d)--(2,2)--(1+\d,2)
(2,0+\d)--(2,1)--(1+\d,1)
(2,-3+\d)--(2,-2)--(2,-1-\d)
(3-\d,-1)--(2,-1)--(2,0-\d)
(1+\d,0)--(2,0)--(3-\d,0)
(3-\s,0)--(4-\d,0)
(0+\d,0)--(1,0)--(1,1-\d)
(0+\d,1)--(1,1)--(1,2-\d)
(0+\d,2)--(1,2)--(1,3-\d)
(1+\d,3)--(2,3)--(2,4-\d)
(0+\d,3)--(1,3)--(1,4-\d)
(-1-\d,3)--(-2,3)--(-2,4-\d)
(0-\d,3)--(-1,3)--(-1,4-\d)
(0,2+\d)--(0,3)--(0,4-\d)
(0,1+\d)--(0,2)--(-1+\d,2)
(0,0+\d)--(0,1)--(-1+\d,1)
(-1,1+\d)--(-1,2)--(-2+\d,2)
(-1,0+\d)--(-1,1)--(-2+\d,1)
(-2,1-\d)--(-2,0)--(-1-\d,0)
(0-\d,0)--(-1,0)--(-1,-1+\d)
(0,0-\d)--(0,-1)--(-1+\d,-1)
(0,-1-\d)--(0,-2)--(-1+\d,-2)
(-1,-1+\s)--(-1,-2+\d)
(1+\d,-3)--(2,-3)--(3-\d,-3)
(0+\d,-3)--(1,-3)--(1,-4+\d)
(-2+\d,-2)--(-1,-2)--(-1,-3+\d)
(-1,-4+\d)--(-1,-3)--(0-\d,-3)
(0,-2-\d)--(0,-3)--(0,-4+\d)
(-3+\d,-1)--(-2,-1)--(-2,-2+\d)
(-3+\d,-2)--(-2,-2)--(-2,-3+\d)
(-3+\d,-3)--(-2,-3)--(-2,-4+\d)
(-4+\d,-1)--(-3,-1)--(-3,-2+\d)
(-4+\d,-2)--(-3,-2)--(-3,-3+\d)
(-4+\d,-3)--(-3,-3)--(-3,-4+\d)
(4-\d,-1)--(3,-1)--(3,-2+\d)
(4-\d,-2)--(3,-2)--(3,-3+\d)
(4-\d,-3)--(3,-3)--(3,-4+\d)
(4-\d,3)--(3,3)--(3,2+\d)
(3-\s,2)--(4-\d,2)
(-3+\d,2)--(-2,2)--(-2,1+\d)
(-2+\s,1)--(-3+\d,1)
(-4+\d,1)--(-3,1)--(-3,2-\d)
(-3,2-\s)--(-3,3-\d)
(-3,3-\s)--(-3,4-\d)
;

\draw[s-helper]
(4,4)--(4,3) (4,1)--(4,0) (4,0)--(4,-1) (4,-2)--(4,-3) (4,-3)--(4,-4)
(3,3)--(3,2) (3,1)--(3,0) (3,-1)--(3,-2) (3,-2)--(3,-3) (3,-3)--(3,-4)
(2,4)--(2,3) (2,2)--(2,1) (2,1)--(2,0) (2,0)--(2,-1) (2,-1)--(2,-2) (2,-2)--(2,-3)
(1,4)--(1,3) (1,3)--(1,2) (1,2)--(1,1) (1,1)--(1,0) (1,0)--(1,-1) (1,-1)--(1,-2) (1,-3)--(1,-4)
(0,4)--(0,3) (0,3)--(0,2) (0,2)--(0,1) (0,1)--(0,0) (0,0)--(0,-1) (0,-1)--(0,-2) (0,-2)--(0,-3) (0,-3)--(0,-4)
(-1,4)--(-1,3) (-1,2)--(-1,1) (-1,1)--(-1,0) (-1,0)--(-1,-1) (-1,-1)--(-1,-2) (-1,-2)--(-1,-3) (-1,-3)--(-1,-4)
(-2,4)--(-2,3) (-2,2)--(-2,1) (-2,1)--(-2,0) (-2,-1)--(-2,-2) (-2,-2)--(-2,-3) (-2,-3)--(-2,-4)
(-3,4)--(-3,3) (-3,3)--(-3,2) (-3,2)--(-3,1) (-3,0)--(-3,-1) (-3,-1)--(-3,-2) (-3,-2)--(-3,-3) (-3,-3)--(-3,-4)

(3,4)--(4,4) (2,4)--(3,4) (1,4)--(2,4) (-2,4)--(-1,4) (-4,4)--(-3,4)
(3,3)--(4,3) (1,3)--(2,3) (0,3)--(1,3) (-1,3)--(0,3) (-2,3)--(-1,3)
(3,2)--(4,2) (1,2)--(2,2) (0,2)--(1,2) (-1,2)--(0,2) (-2,2)--(-1,2) (-3,2)--(-2,2)
(2,1)--(3,1) (1,1)--(2,1) (0,1)--(1,1) (-1,1)--(0,1) (-2,1)--(-1,1) (-3,1)--(-2,1) (-4,1)--(-3,1)
(3,0)--(4,0) (2,0)--(3,0) (1,0)--(2,0) (0,0)--(1,0) (-1,0)--(0,0) (-2,0)--(-1,0)
(3,-1)--(4,-1) (2,-1)--(3,-1) (0,-1)--(1,-1) (-1,-1)--(0,-1) (-3,-1)--(-2,-1) (-4,-1)--(-3,-1)
(3,-2)--(4,-2) (0,-2)--(1,-2) (-1,-2)--(0,-2) (-2,-2)--(-1,-2) (-3,-2)--(-2,-2) (-4,-2)--(-3,-2)
(3,-3)--(4,-3) (2,-3)--(3,-3) (1,-3)--(2,-3) (0,-3)--(1,-3) (-1,-3)--(0,-3) (-3,-3)--(-2,-3) (-4,-3)--(-3,-3)
;

\foreach \x / \y in
{
0/0,-4/4,-4/3,-4/2,-4/1,-4/0,-4/-1,-4/-2,-4/-3,-4/-4,
-3/4,-2/4,-1/4,0/4,1/4,2/4,3/4,4/4,4/3,4/2,4/1,4/0,4/-1,
4/-2,4/-3,4/-4,3/-4,2/-4,1/-4,0/-4,-1/-4,-2/-4,-3/-4
}
\filldraw[xshift=\x cm, yshift=\y cm, fill=black] (0,0) circle (3pt);

\draw
(-3,0) node[s-node] {01}
(1,-2) node[s-node] {02}
(1,-1) node[s-node] {03}
(3,1) node[s-node] {04}
(2,2) node[s-node] {05}
(2,1) node[s-node] {06}
(2,-2) node[s-node] {07}
(2,-1) node[s-node] {08}
(2,0) node[s-node] {09}
(3,0) node[s-node] {10}
(1,0) node[s-node] {11}
(1,1) node[s-node] {12}
(1,2) node[s-node] {13}
(2,3) node[s-node] {14}
(1,3) node[s-node] {15}
(-2,3) node[s-node] {16}
(-1,3) node[s-node] {17}
(0,3) node[s-node] {18}
(0,2) node[s-node] {19}
(0,1) node[s-node] {20}
(-1,2) node[s-node] {21}
(-1,1) node[s-node] {22}
(-2,0) node[s-node] {23}
(-1,0) node[s-node] {24}
(0,-1) node[s-node] {25}
(0,-2) node[s-node] {26}
(-1,-1) node[s-node] {27}
(2,-3) node[s-node] {28}
(1,-3) node[s-node] {29}
(-1,-2) node[s-node] {30}
(-1,-3) node[s-node] {31}
(0,-3) node[s-node] {32}
(-2,-1) node[s-node] {33}
(-2,-2) node[s-node] {34}
(-2,-3) node[s-node] {35}
(-3,-1) node[s-node] {36}
(-3,-2) node[s-node] {37}
(-3,-3) node[s-node] {38}
(3,-1) node[s-node] {39}
(3,-2) node[s-node] {40}
(3,-3) node[s-node] {41}
(3,3) node[s-node] {42}
(3,2) node[s-node] {43}
(-2,2) node[s-node] {44}
(-2,1) node[s-node] {45}
(-3,1) node[s-node] {46}
(-3,2) node[s-node] {47}
(-3,3) node[s-node] {48}
;
\end{tikzpicture}
\caption{Clumps of size $1$ and $2$ produced by \Cref{alg:core-construction} on a configuration with $p = 0.8$. Vertices are numbered as they appear.}
\label{fig:clumps}
\end{figure}
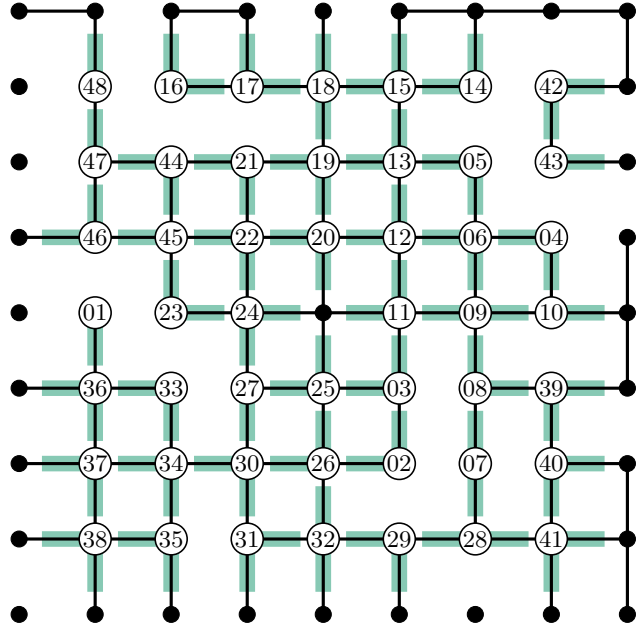

As $v_0$ is $m$-confined, $F$ has no path from $v_0$ to $\vboundary B_m(v_0)$ by the end of the algorithm.
Indeed, let $H$ to be the component of $v_0$ in $F$.
Then $H$ is a connected open subgraph of $B_m(v_0)$ with all internal vertices $x$ satisfying $\deg_H(x) \geq d + 1$, so the condition that $v_0$ is $m$-confined guarantee that $H \cap \vboundary B_m(v_0) = \emptyset$.
Therefore, any open path on $B_m(v_0)$ must use a vertex in $U$.
See \Cref{fig:clumps} for an example of the outcome of \Cref{alg:core-construction}.

We now can formally describe Breaker's strategy for a confined configuration, together with the family of clumps $G$ just constructed above.
On any given turn, if Maker claims the edge $e$, then Breaker claims edges in the following order:
\begin{enumerate}
    \item \label{strat:11breaker-step1}
    If $e \in S$ for some clump $S \in G$, claim all other unclaimed edges in $S$.
    \item \label{strat:11breaker-step2}
    Claim any unclaimed open edges inside $B_m(v_0)$.
    \item \label{strat:11breaker-step3}
    Claim any unclaimed open edges anywhere.
\end{enumerate}
On Breaker's turn, he has $d-1$ edges to claim, and he starts by following rule \ref{strat:11breaker-step1}.
If indeed Maker's edge $e$ belongs to a clump $S$, then it must be the first edge in to be claimed in that clump, and moreover, as $2 \leq \card{S} \leq d$, Breaker indeed can claim all other edges in $S$.
We say now that the clump $S$ has been \indef{used}.
If Breaker can still claim edges, he claims any available edge following \ref{strat:11breaker-step2} and \ref{strat:11breaker-step3}, giving priority to edges inside $B_m(v_0)$.
On finitely many rounds, all the open edges in $B_m(v_0)$ are claimed, as Breaker only claims such edges when they exists.

\begin{proof}[Proof of \Cref{prop:confined-strategy}]
We show that if Breaker follows the strategy just described, then it wins on a confined configuration.
Indeed, by the time that all open edges in $B_m(v_0)$ are claimed, Breaker has already won.
It suffices then to show that there are no self-avoiding walk $P$ from $v_0$ to $\vboundary B_m(v_0)$ consisting only of edges claimed by Maker.
Suppose such path $P$ existed, we may assume it is fully contained in $B_m(v_0)$.
As previously observed, any such path must use a vertex in $U = {u_1, u_2, \dotsc}$.
Let $u_k$ be the vertex of smallest index with $u_k \in P$.
As $U$ consists only of internal vertices, $u_k$ is an internal vertex of $P$.
Let $x$ and $y$ be the two neighbours of $u_k$ in $P$.
From \Cref{alg:core-construction}, $u_k$ can be one of two kinds: it is either a centre of a clump or it is not the centre of a clump.

If $u_k$ is not the centre of a clump, then $u_k$ had degree $1$ in the $k$-th step of the algorithm, which means that it has only one neighbour in $B_m(v_0) \setminus \set{u_1, \dotsc, u_k}$.
This contradicts the minimality of $k$, as it implies either $x$ or $y$ is in $\set{u_1, \dotsc, u_k}$.

If $u_k$ is the centre of a clump $S_k$, then at the $k$-th step of the algorithm $u_k$ had from $2$ to $d$ neighbours in $B_m(v_0) \setminus \set{u_1, \dotsc, u_k}$.
Since all other vertices of $P$, and in particular $x$ and $y$, are not in $\set{u_1, \dotsc, u_k}$, we have that the edges $u_k x$ and $u_k y$ are in $S_k$.
This is a contradiction as Breaker claimed at least one of these two edges.
\end{proof}

\subsection{Bootstrap percolation}

The \indef{$r$-neighbour bootstrap percolation} of a graph $G$ is the following deterministic process.
An initial set of vertices $I \subseteq V(G)$ is deemed \indef{infected}.
Writing $I_t$ for the set of infected vertices at time $t$, we set $I_0 \defined I$ and follow the dynamics
\begin{equation*}
  I_{t+1} \defined I_t \cup \set[\big]{v \in V(G) \st \card{N(v) \cap I_t} \geq r}.
\end{equation*}
In other words, a vertex becomes infected as soon as it has at least $r$ infected neighbours.
Also note that a vertex cannot become disinfected, so the sets $I_t$ are nested.
The \indef{closure} of $I \subseteq V(G)$ is the set of eventually infected vertices under $r$-neighbour bootstrap percolation, that is, $[I]_r \defined \bigcup_{t \geq 0} I_t$.
We say that the set $I$ \indef{percolates} if $[I]_r = V(G)$.
While deterministic initial conditions are also of interest, we are concerned with the case that $I$ is chosen randomly.
Specifically for $\ZZ^d$, when the initially infected set is the result of a site percolation process with $p > 0$, percolation always occurs under the $d$-neighbour bootstrap.
Indeed, this was shown by van Enter~\cite{Van_Enter1987-ee} for $d = 2$ and by Schonmann~\cite{Schonmann1992-lj} for general $d$.

\begin{theorem}[van Enter and Schonmann]
\label{thm:schonmann-van-enter}
If $0 < p \leq 1$ and $0 \leq r \leq d$ then under the $r$-neighbour bootstrap percolation, the set $(\ZZ^d)_p^{\site}$ percolates almost surely.
\end{theorem}

Using this result, we can now show that the $(d+1)$-core of the open edges in bond percolation in $\ZZ^d$ is empty.

\begin{proposition}
\label{prop:minimum-degree-percolation}
For any $0 \leq p < 1$, almost surely $(\ZZ^d)_p$ has no open subgraph of minimum degree $\geq d + 1$.
\end{proposition}
\begin{proof}
Given a vertex $x \in \ZZ^d$ denote by $S_x$ the set of edges of the form $\set{x, x + e_i}$, where $e_1, \dotsc, e_d$ is the standard basis of $\RR^d$.
Consider the set
\begin{equation*}
    I_0 \defined \set[\big]{x \in \ZZ^d \st \text{all edges in $S_x$ are closed}}.
\end{equation*}
Note that each vertex $x \in \ZZ^d$ is in $I_0$ independently with probability $q = (1 - p)^d > 0$.
In other words, $I_0$ has the law of $\Psite{\ZZ^d,q}$.
\Cref{thm:schonmann-van-enter} gives that the $d$-neighbour bootstrap percolation on $\ZZ^d$ with initially infected vertices $I_0$ does percolate almost surely.

We now show that $[I_0]_d = \ZZ^d$ implies that there are no open subgraphs of $\ZZ^d$ with minimum degree $\geq d + 1$.
Indeed, assume we had such a subgraph $H$.
No vertex in $I_0$ is in $H$, as by definition, vertices in $I_0$ have at most $d$ open neighbours.
Now we observe that no vertex of $H$ can ever get infected in further rounds.
Indeed, say that no vertex of $H$ has been infected after $t-1$ rounds.
If $x \in V(H)$ gets infected in the $t$-th round, then $x$ has at least $d$ neighbours in $I_0 \cup I_1 \cup \dotsb \cup I_{t-1}$, therefore it has at most $d$ neighbours in $H$.
\end{proof}

In another words, \Cref{prop:minimum-degree-percolation} says that
for every $p < 1$,
\begin{equation*}
    \pbond[\big]{\ZZ^d,p}{\text{$(d+1)$-core of $(\ZZ^d)_p$ is empty}} = 1,
\end{equation*}
Using Harris' inequality, we strengthen \Cref{prop:minimum-degree-percolation} as follows.

\begin{proposition}
\label{prop:minimum-degree-percolation-origin}
For any $0 \leq p < 1$, almost surely $(\ZZ^d)_p$ has no open subgraph $H$ with $\card{V(H)} \geq 2$ and $\deg_H(x) \leq d$ for at most one $x \in V(H)$.
\end{proposition}
\begin{proof}
If $p = 0$, there is nothing to prove, so assume $0 < p < 1$.
Let $E_v$ be the event that $(\ZZ^d)_p$ has an open subgraph $H$ with $v \in V(H)$, $\card{V(H)} \geq 2$ and $\deg_H(x) \geq d+1$ for all $x \in V(H)$, $x \neq v$.
Our goal is to show that $\pbond{p}{E} = 0$, where $E = \bigcup_{v \in \ZZ^d}E_v$ is the event in the statement.
It is enough to show that $\pbond{p}{E_0} = 0$, as $\pbond{p}{E_v} = \pbond{p}{E_0}$ for all $v \in \ZZ^d$ by translational invariance.
Assume for a contradiction that $\pbond{p}{E_0} = q > 0$.

Let $e_1, \dotsc, e_{2d}$ be the edges incident to $0$ in $\ZZ^d$ and let $F_i$ be the event that edge $e_i$ is open.
Note that the events $E_0$ and $F_i$ are all increasing.
From Harris's inequality,
\begin{equation*}
    \pbond{p}{E_0 \cap F_1 \cap \dotsb \cap F_{2d}} \geq \pbond{p}{E_0} \pbond{p}{F_1} \dotsb \pbond{p}{F_{2d}} = qp^{2d} > 0.
\end{equation*}
But this contradicts \Cref{prop:minimum-degree-percolation}, as $E_0 \cap F_1 \cap \dotsb \cap F_{2d}$ implies that there is an open subgraph of minimum degree $\geq d + 1$.
\end{proof}

The last step in the proof of \Cref{thm:1dtrivial} is to show that $(\ZZ^d)_p$ is confined almost surely.
This follows from \Cref{prop:minimum-degree-percolation-origin} and a compactness argument.

\begin{proposition}
\label{prop:confined}
For any $0 \leq p < 1$, $(\ZZ^d)_p$ is confined almost surely.
\end{proposition}
\begin{proof}
Fix $v \in \ZZ^d$.
By \Cref{prop:minimum-degree-percolation-origin}, we can assume that $(\ZZ^d)_p$ almost surely has no open subgraph $H$ with $v \in V(H)$, $\card{V(H)} \geq 2$ and $\deg_H(x) \leq d$ for all $x \in V(H)$, $x \neq v$.
We will show that this is imply that $v$ is $m$-confined in $(\ZZ^d)_p$.

Enumerate all the vertices of $\ZZ^d \setminus \set{v}$ as $v_1, v_2, \dotsc$ such that $d(v,v_i) \leq d(v,v_j)$ whenever $i \leq j$, where $d$ is the graph distance in $\ZZ^d$.
If there is no $m \geq 2$ such that $v$ is $m$-confined in $\ZZ^d$, there is an infinite collection $\cH$ of open subgraphs $H_2, H_3, \dotsc$ where $v \in V(H_m)$, $\card{V(H_m) \cap \vboundary B_m(v)} \geq 1$ and $\deg_{H_m}(x) \geq d + 1$ for any every internal vertex $x \in V(H_m)$.
We can assume that each $H_m$ is induced by a set of vertices $V_m$.

We iteratively construct an induced subgraph $H$ in the following way.
If $v_1$ is infinitely many subgraphs of $\cH$ then put $v_1$ in $V(H)$ and pass to the subfamily of $\cH$ consisting of all subgraphs with $v_1$.
Otherwise, $v_1$ is absent in infinitely many subgraphs of $\cH$, so we pass to that subfamily.
Repeat this process for $v_2, v_3, \dotsc$ and so on to finish the construction of $H$.

We claim that $H$ is an open subgraph where each vertex, possibly apart from $v$, has degree at least $3$.
Indeed, let $x$ be a vertex in the interior of $B_m(v)$ for some $m$.
As the vertices $v_1, v_2, \dotsc$ cover all of $\ZZ^d \setminus \set{v}$, each vertex of $B_m(v)$ has been either placed or not in $H$.
As $x \in H$, then $\cH$ was eventually restricted to a family of subgraphs that all agreed on $B_m(v)$, and $\deg_{H'}(x) \geq 3$ for all $H' \in \cH$ from that point onward.
\end{proof}

\Cref{thm:1dtrivial} and \Cref{thm:11trivial} are now proven.


\section{Maker-Breaker games on hypergraphs}
\label{sec:hyper}

A \indef{hypergraph} $\cH = (V, E)$ is a set $V = V(\cH)$ of \indef{vertices}, together with a family $E = E(\cH)$ of subsets of $V$ called \indef{hyperedges}.
It is often convenient to consider Maker-Breaker games played on an arbitrary hypergraph $\cH$.
In this setting, the board of the game consists of the vertex set of $\cH$ while the winning sets are its hyperedges.

We define a \indef{game configuration} on $\cH$ to be a map $\sigma \from V(\cH) \to \set{\rU, \rM, \rB}$.
Given a configuration $\sigma$, we say that an element of the board, represented by a vertex $v \in V(\cH)$, is \indef{of Maker} if $\sigma(v) = \rM$, \indef{of Breaker} if $\sigma(v) = \rB$, and \indef{unclaimed} if $\sigma(v) = \rU$.
We write $\rM_\sigma \defined \sigma^{-1}(\rM)$ for the set of elements claimed by Maker, $\rB_\sigma \defined \sigma^{-1}(\rB)$ for the set of elements claimed by Breaker and $\rU_\sigma \defined \sigma^{-1}(\rU)$ for the set of unclaimed elements.

Given a game configuration $\sigma$ and an unclaimed element $v \in \rU_\sigma$, denote by $\sigma_{\rM \gets v}$ the game configuration obtained from $\sigma$ by giving the element $v$ to Maker.
We define $\sigma_{\rB \gets v}$ similarly.
That is, we set
\begin{align*}
    \sigma_{\rM \gets v}(u) \defined \begin{cases}
        \rM & \text{if $u = v$}, \\
        \sigma(u) & \text{otherwise},
    \end{cases}
    & &
    \sigma_{\rB \gets v}(u) \defined \begin{cases}
        \rB & \text{if $u = v$}, \\
        \sigma(u) & \text{otherwise}.
    \end{cases}
\end{align*}
Given a game configuration $\sigma$, the \indef{reverse configuration} $\sigma^{\rR}$ is the one obtained from $\sigma$ by reversing the ownership of the elements claimed by Breaker and Maker.
That is,
\begin{equation*}
    \sigma^{\rR}(v) \defined \begin{cases}
        \rM & \text{if $\sigma(v) = \rB$}, \\
        \rB & \text{if $\sigma(v) = \rM$}, \\
        \rU & \text{if $\sigma(v) = \rU$}.
    \end{cases}
\end{equation*}
Observe that if $\sigma$ is distributed as $(\Lambda)_{\alpha, \beta}$, then $\sigma^{\rR}$ is distributed as $(\Lambda)_{\beta, \alpha}$.

Given integers $m, b \geq 1$ and a game configuration $\sigma_0$ on a hypergraph $\cH$ we define the \indef{$(m,b)$ Maker-Breaker game on $\cH$ starting from $\sigma_0$} in the following way.
Players Maker and Breaker alternate turns claiming unclaimed elements of the board $V(\cH)$ and updating the game configuration accordingly.
The initial configuration is $\sigma_0$ and we write $\sigma_i$ for the configuration after $i$ elements have been claimed since the start of the game.
Maker starts and on her turn she claims $m$ unclaimed elements, while Breaker claims $b$ on his turn.
In this way, a sequence $\sigma_0, \sigma_1, \sigma_2, \dotsc$ of game configurations is produced throughout the game, with either $\sigma_{i+1} = (\sigma_i)_{\rM \gets v}$ or $\sigma_{i+1} = (\sigma_i)_{\rB \gets v}$, where $v \in \rU_{\sigma_{i}}$.
If we set $g = m + b$, then the game configuration at the beginning of round $k$ is $\sigma_{gk}$.
The winning sets are the hyperedges $E(\cH)$ and Breaker's goal is to claim at least one element in each winning set.
In other words, Breaker wins in round $k$ if for every $S \in E(\cH)$, we have $\card{S \cap \rB_{\sigma_{gk}}} \geq 1$.
Maker wins if there exists $S \in E(\cH)$ such that $\card{S \cap \rB_{\sigma_i}} = 0$ for every $i \geq 0$.
That is to say, a winning strategy for Maker is a way to guarantee that Breaker will never win in finite time.
If $V(\cH)$ is finite, every element is eventually claimed, so Maker can only win by fully claiming a hyperedge $S \in E(\cH)$.

When we do not mention a particular initial configuration $\sigma_0$, we consider the game configuration where every element is unclaimed, that is, $\sigma_0(v) = \rU$ for all $v \in V(\cH)$.
However, it is important for our applications to consider games with various initial configurations $\sigma_0$.
In fact, we will often consider the case where $\sigma_0$ is random.
We note that any Maker-Breaker game can be modelled as a Maker-Breaker game on a hypergraph.
Indeed, one can always consider the hypergraph where the elements of the board and the winning sets are its vertices and hyperedges, respectively.
This hypergraph is sometimes referred in the literature to as the \indef{hypergraph of the game} (see, e.g., \cite{Hefetz2014-qc}).

The seminal paper of Erd\H{o}s and Selfridge~\cite{Erdos1973-gg} introduced a potential-based strategy and gave a general criterion for a win of Breaker in the unbiased Maker-Breaker game on a finite hypergraph $\cH$.
They showed that if
\begin{equation*}
    \sum_{S \in E(\cH)} 2^{-\card{S}} < \frac{1}{2}
\end{equation*}
then Breaker has a winning strategy for the $(1,1)$ Maker-Breaker on $\cH$.
This result was subsequently generalised by Beck~\cite{Beck1982-rh}, who noted that a similar potential-based strategy works for biased Maker-Breaker games.

\begin{theorem}[Beck]
\label{thm:beck-original}
Let $\cH$ be a finite hypergraph and let $m,b \geq 1$ be integers.
If
\begin{equation}
\label{eq:beck-condition}
    \sum_{S \in E(\cH)} (b+1)^{-\card{S}/m} < \frac{1}{b+1},
\end{equation}
then Breaker has a winning strategy for the $(m,b)$ Maker-Breaker game on $\cH$.
\end{theorem}

While we will not use \Cref{thm:beck-original}, the following version of Beck's result will be more adequate for our applications.

\begin{theorem}
\label{thm:beck-infinite}
Let $\cH$ be a hypergraph and let $m, b \geq 1$ be integers.
Let $\sigma_0$ be a game configuration on $\cH$.
If
\begin{equation}
\label{eq:beck-infinite-condition}
    \sum_{\substack{S \in E(\cH)\\S \cap \rB_{\sigma_0} = \emptyset}} (b+1)^{-\card{S \cap \rU_{\sigma}}/m} < \frac{1}{b+1},
\end{equation}
then Breaker has a strategy for the $(m,b)$ Maker-Breaker game on $\cH$ starting from $\sigma_0$ which guarantees that Maker never fully claims a winning set $S \in E(\cH)$.
In particular, Breaker wins if $\cH$ is finite.
\end{theorem}

We must emphasise that the proof of \Cref{thm:beck-original} works for \Cref{thm:beck-infinite} as well.
In fact, the difference between the conditions \eqref{eq:beck-condition} and \eqref{eq:beck-infinite-condition} is purely cosmetic.
Indeed, \eqref{eq:beck-infinite-condition} is the same as \eqref{eq:beck-condition} on the auxiliary hypergraph $\cH'$ obtained from $\cH$ by deleting all hyperedges $S$ with $S \cap \rB_\sigma \neq \emptyset$, and by removing all vertices $v$ with $\sigma(v) = \rM$ from each remaining hyperedge.
The fact that we are considering an infinite hypergraph presents no real difficulty as long as the sum in \eqref{eq:beck-infinite-condition} is finite.
For the benefit of the reader, we include a detailed proof of \Cref{thm:beck-infinite} in \Cref{sec:analysis} as we believe it is instructive to see the analysis of the potential-based strategy.

We also remark that the proof of \Cref{thm:beck-infinite} can be trivially extended to handle $c$-boosted games, as long as the right-hand side of \eqref{eq:beck-infinite-condition} is strengthened from $1/(b+1)$ to $1/(b+1)^{(m+c)/m}$, see \Cref{cor:beck-infinite-boosted}.

In order to use \Cref{thm:beck-infinite} in the analysis of the Maker-Breaker percolation games, we consider an appropriate auxiliary hypergraph $\cH$.
Generally speaking, there are two different hypergraphs that will be relevant for us, depending on whether we want to analyse a strategy for Maker or for Breaker.

If $\Lambda$ is a connected locally finite\footnote{A graph is locally finite if all the vertices have finite degrees.} graph, $v_0 \in V(\Lambda)$ and $N \geq 0$, then we introduce a hypergraph $\cP_{\Lambda,v_0}^N$ defined as follows.
Consider all the self-avoiding walks in $\Lambda$ that start from $v_0$ and have length exactly $N$.
Each of these walks corresponds to a separate hyperedge in $\cP_{\Lambda,v_0}^N$ formed by all the edges of the walk.
The vertex set of $\cP_{\Lambda,v_0}^N$ consists of all the edges of $\Lambda$ that are contained in at least one such self-avoiding walk.
As $\Lambda$ is locally finite, $\cP_{\Lambda,v_0}^N$ has finitely many vertices.
We use $\cP_{\Lambda,v_0}^N$ for the analysis of Breaker's winning strategy, since if he can claim one edge in each self-avoiding walk of length $N$ from $v_0$ he wins the Maker-Breaker percolation game.

\begin{observation}
\label{obs:breaker-hypergraph}
Suppose that $\sigma$ is a game configuration on $\Lambda$ with the property that for every $v_0 \in V(\Lambda)$, there is $N$ such that Breaker has a winning strategy for the $(m,b)$ Maker-Breaker game on $\cP_{\Lambda,v_0}^N$.
Then Breaker has a winning strategy for the $(m,b)$ Maker-Breaker percolation game on $\Lambda$ with initial configuration $\sigma$.
\end{observation}

We now define a hypergraph which will be useful for the analysis of Maker's win.
Let $\Lambda$ be a connected infinite graph and $v_0 \in V(\Lambda)$, then the hypergraph $\cC_{\Lambda,v_0}$ is defined as follows.
The vertices of $\cC_{\Lambda,v_0}$ are the edges $E(\Lambda)$.
The hyperedges of $\cC_{\Lambda,v_0}$ are all the minimal finite subsets of $E(\Lambda)$ whose removal separates $v_0$ from the infinite component.
When $\Lambda$ is a plane graph, this can be better understood in terms of the dual graph $\Lambda'$.
Indeed, a game configuration $\sigma \from E(\Lambda) \to \set{\rM, \rB, \rU}$ in $\Lambda$ can be associated with a dual configuration $\sigma' \from E(\Lambda') \to \set{\rM, \rB, \rU}$ characterised by $\sigma_{e} = (\sigma')_{e'}$.
Then, if $C \in \cC_{\Lambda,v_0}$ is a hyperedges, then the dual of the edges in $C$ are edges in self-avoiding dual cycle around $v_0$.
Maker can exploit this duality and play in $\Lambda$ as Breaker would play in $\Lambda'$ in the hypergraph with the dual cycles as hyperedges.
We prefer, however, to keep the game being played on the original graph $\Lambda$ and freely use the language of the dual graph when it is clearer to do so.

\begin{observation}
\label{obs:maker-hypergraph}
Suppose that $\sigma$ is a game configuration on $\Lambda$ where there is $v_0 \in V(\Lambda)$ such that Breaker has a strategy for the $(b,m)$ Maker-Breaker game on $\cC_{\Lambda,v_0}$ in which Maker never fully claims a winning set.
Then Maker has a winning strategy for the $(m,b)$ Maker-Breaker percolation game on $\Lambda$ with initial configuration $\sigma^{\rR}$.
\end{observation}

In other words, if Breaker almost surely has a winning strategy for the $(b, m)$ Maker-Breaker game on $\cC_{\Lambda,v_0}$ with initial configuration drawn from $(\Lambda)_{\beta, \alpha}$, then Maker almost surely has a winning strategy for the $(m, b)$ game on $(\Lambda)_{\alpha, \beta}$.

Due to technical reasons that will arise later, we restrict our attention to the case $\Lambda = \ZZ^2$ when considering the potential strategy for Maker.
We simplify the notation and write $\cC_{v_0}$ for $\cC_{\ZZ^2,v_0}$ moving forward.
In this case, the hyperedges of $\cC_{v_0}$ correspond to cycles in the plane dual graph of $\ZZ^2$.

In \Cref{sec:potential}, we will use the hypergraphs $\cP_{\Lambda,v_0}^N$ and $\cC_{\ZZ^2,v_0}$ to analyse winning strategies for Maker and Breaker for a variety of Maker-Breaker percolation games.
While \Cref{obs:breaker-hypergraph} gives us a successful approach for Breaker's win, \Cref{obs:maker-hypergraph} does not quite work for Maker if there is an abundance of short dual cycles.
We fix this issue by showing how Maker can guarantee that Breaker cannot fully claim a very short dual cycle in an auxiliary game.


\section{The potential-based strategies}
\label{sec:potential}

Given a connected infinite vertex-transitive graph $\Lambda$ and parameters $\alpha, \beta \geq 0$, $\alpha + \beta \leq 1$, we consider potential-based strategies for the $(m,b)$ game on $(\Lambda)_{\alpha,\beta}$.
In view of \Cref{obs:breaker-hypergraph} and \Cref{obs:maker-hypergraph}, we will use potential-based strategies for both players in the Maker-Breaker games on the auxiliary hypergraphs $\cP_{\Lambda,v_0}^N$ and $\cC_{\Lambda,v_0}$.

\subsection{Winning the game as Breaker}

To win the $(m,b)$ Maker-Breaker percolation game on $\Lambda$, Breaker has to claim one edge in every long enough self-avoiding walk in $\Lambda$.
Indeed, given an initial configuration $\sigma_0$ and an origin $v_0 \in V(\Lambda)$, Breaker can choose a suitable length $N$ and hope he can claim one edge in each self-avoiding walk from $v_0$ of length $N$.
In other words, Breaker can play the $(m,b)$ Maker-Breaker game on $\cP_{\Lambda,v_0}^N$.
To this effect, it is sufficient to check that condition \eqref{eq:beck-infinite-condition} in \Cref{thm:beck-infinite} holds.

The core idea of a potential-based strategy is that we provide Breaker with a way of evaluating how `urgent' it is to play an edge $e \in E(\Lambda)$ given a game configuration $\sigma$.
Breaker then prioritises playing the most `urgent' edges.
Given a self-avoiding walk $P \in E(\cP_{\Lambda,v_0}^N)$ and a game configuration $\sigma$, define
\begin{equation*}
    r(\sigma, P) \defined
    \begin{cases}
        \card{P \cap \rU_\sigma} & \text{if $P \subseteq \rU_\sigma \cup \rM_\sigma$}, \\
        \infty & \text{otherwise,}
    \end{cases}
\end{equation*}
and for a fixed $0 < \lambda < 1$ to be chosen later, define the \indef{danger} of a path $P$ to be
\begin{equation*}
    w_\lambda(\sigma, P) \defined  \lambda^{r(\sigma, P)}.
\end{equation*}
The total danger associated with a configuration $\sigma$ is then defined as
\begin{equation*}
    w_\lambda(\sigma, \cP_{\Lambda,v_0}^N) \defined \sum_{P \in E(\cP_{\Lambda,v_0}^N)} w_\lambda(\sigma, P).
\end{equation*}
Breaker's goal is to keep $w_\lambda$ as low as possible.
Indeed, if $w_\lambda(\sigma, \cP_{\Lambda,v_0}^N)$ is less than $1$ throughout the whole game, Breaker wins.
See \Cref{sec:analysis} for the detailed analysis.

Of course, there is no guarantee that the strategy of greedily minimising $w_\lambda$ is successful.
From \Cref{thm:beck-infinite}, together with \Cref{obs:breaker-hypergraph}, we know that as long as $\lambda = (b+1)^{-1/m}$, if regardless of the choice of $v_0$ there is an $N = N(v_0)$ with $w_\lambda(\sigma_0, \cP_{\Lambda,v_0}^N) < 1/(b+1)$, then the greedily minimising $w_\lambda$ is a winning strategy for Breaker for the $(m,b)$ Maker-Breaker percolation game on $\Lambda$ starting from $\sigma_0$.
It is left for us to analyse how small can $w_\lambda(\sigma_0, \cP_{\Lambda,v_0}^N)$ be when $\sigma_0$ is drawn from $(\Lambda)_{\alpha,\beta}$.

\begin{proposition}
\label{prop:weight-saw}
Let $\Lambda$ be an infinite vertex-transitive graph with a connective constant bounded from above by $\kappa$. 
Let $\alpha, \beta \geq 0$, $\alpha + \beta \leq 1$ and let $0 < \lambda < 1$ be such that $\p[\big]{\lambda(1 - \alpha - \beta) + \alpha} \kappa < 1$.
Then almost surely, a random game configuration $\sigma$ distributed as $(\Lambda)_{\alpha,\beta}$ has the property that for any $v_0 \in V(\Lambda)$ and any $\delta > 0$ , there exists an integer $N$ for which $w_\lambda(\sigma, \cP_{\Lambda,v_0}^N) < \delta$.
\end{proposition}
\begin{proof}
Let $\sigma$ be a random game configuration distributed as $(\Lambda)_{\alpha,\beta}$.
That is, every edge $e \in E(\Lambda)$ is set independently to $e \in \rM_\sigma$ with probability $\alpha$, to $e \in \rB_\sigma$ with probability $\beta$ and to $e \in \rU_\sigma$ with probability $1 - \alpha - \beta$.

Since $\Lambda$ is vertex-transitive, we have $\card{E(\cP_{\Lambda,v}^{N + M})} \leq \card{E(\cP_{\Lambda,v}^N)} \card{E(\cP_{\Lambda,v}^M)}$, therefore $\card{E(\cP_{\Lambda,v}^N)}^{1/N}$ converges to the \indef{connective constant} of $\Lambda$.
Since $\kappa$ is an upper bound on the connective constant, for every $\eps > 0$ there exists $N_0(\eps)$ such that $\card{\cP_{\Lambda,v}^N} \leq (\kappa + \eps)^N$ for every $N \geq N_0(\eps)$.
As $\kappa < 1 / \p[\big]{\lambda(1 - \alpha - \beta) + \alpha}$, we can choose $\eps > 0$ such that $\kappa + \eps < 1 / \p[\big]{\lambda(1 - \alpha - \beta) + \alpha}$.

Note that $w_\lambda(\sigma,P) = \ind_{\set{P \cap \rB_\sigma = \emptyset}} \lambda^{\card{P \cap \rU_\sigma}}$.
Therefore,
\begin{align*}
    \expec[\big]{w_\lambda(\sigma,P)}
    &= \sum_{0 \leq k \leq N} \lambda^k \pbond[\big]{\Lambda,\alpha,\beta}{\card{P \cap \rU_\sigma} = k, \card{P \cap \rB_\sigma} = 0} \\
    &= \sum_{0 \leq k \leq N} \lambda^k \binom{N}{k}(1 - \alpha - \beta)^k \alpha^{N-k}
    = \p[\big]{\lambda(1 - \alpha - \beta) + \alpha }^N.
\end{align*}
Thus, for $N \geq N_0(\eps)$, we have
\begin{align*}
    \expec[\big]{w_\lambda(\sigma, \cP_{\Lambda,v}^N)}
    &= \sum_{P \in E(\cP_{\Lambda,v}^N)} \expec[\big]{w_\lambda(\sigma, P)} = \card{\cP_{\Lambda,v}^N} \p[\big]{\lambda (1 - \alpha - \beta) + \alpha}^N \\
    &\leq \p[\Big]{(\kappa + \eps)\p[\big]{\lambda (1 - \alpha - \beta) + \alpha}}^N.
\end{align*}

Let $\cE_v^N$ be the event that $w_\lambda(\sigma, \cP_{\Lambda,v}^N) \geq \delta$.
Then $\cE_v = \bigcap_{N \geq 0}\cE_v^N$ is the event that $w_\lambda(\sigma, \cP_{\Lambda,v}^N) \geq \delta$ holds for all $N$.
By Markov's inequality, we have
\begin{equation*}
    \pbond{\Lambda,\alpha,\beta}{\cE_v}
    \leq \pbond{\Lambda,\alpha,\beta}{\cE_v^N}
    \leq \delta^{-1}\p[\big]{(\kappa + \eps)\p[\big]{\lambda(1 - \alpha - \beta) + \alpha}}^N
\end{equation*}
for $N \geq N_0(\eps)$.
As $(\kappa + \eps)\p[\big]{\lambda(1 - \alpha - \beta) + \alpha} < 1$, we have $\pbond{\Lambda,\alpha,\beta}{\cE_v} = 0$ for all $v \in V(\Lambda)$.
Therefore, if $\cE = \bigcup_{v \in V(\Lambda)} \cE_v$ is the event that there is some vertex $v$ such that $w_\lambda(\sigma, \cP_{\Lambda,v}^N) \geq \delta$ for all $N$, then $\pbond{\Lambda,\alpha,\beta}{\cE} \leq \sum_{v \in V(\Lambda)} \pbond{\Lambda,\alpha,\beta}{\cE_v} = 0$, as required.
\end{proof}

Assuming \Cref{thm:beck-infinite}, we obtain \Cref{thm:potential-breaker-symmetric} as a corollary.

\begin{proof}[Proof of \Cref{thm:potential-breaker-symmetric}]
Consider the $c$-boosted $(m,b)$ Maker-Breaker percolation game on $(\Lambda)_{\alpha,\beta}$.
Setting $\lambda \defined (b+1)^{-1/m}$ in the condition $\kappa < 1/\p[\big]{\lambda(1 - \alpha - \beta) + \alpha}$, we obtain \eqref{eq:mb-breaker-ab}.
Let $\sigma_0$ be a random game configuration distributed as $(\Lambda)_{\alpha, \beta}$ and let $v_0$ be the vertex chosen by Maker before the game starts.
Thus, \Cref{prop:weight-saw} gives that almost surely, there exists an integer $N \geq 0$ such that $w_\lambda(\sigma, \cP_{\Lambda,v_0}^N) < 1/(b+1)^{(m+c)/m}$.
With this condition, by \Cref{cor:beck-infinite-boosted}, Breaker has a winning strategy for the $c$-boosted $(m,b)$ Maker-Breaker game on $\cP_{\Lambda, v_0}^N$.
By \Cref{obs:breaker-hypergraph}, we have that almost surely Breaker has a winning strategy for the $c$-boosted $(m, b)$ Maker-Breaker percolation game on $(\Lambda)_{\alpha, \beta}$, as claimed.
\end{proof}

\subsection{Winning the game as Maker}

We now proceed to a similar strategy for Maker when $\Lambda = \ZZ^2$.
In view of \Cref{obs:maker-hypergraph}, we are now concerned with the hypergraph $\cC_{v_0}$.
We further restrict to pairs $(m,b)$ with $m \geq 2$ and $b = 1$.

We maintain the core idea from the strategy of Breaker: we evaluate how `urgent' edges are in need to be played and greedily minimise our estimation of how bad a configuration is.
Indeed, to win as Maker, we play as Breaker in the auxiliary game on $\cC_{v_0}$.
Given a self-avoiding dual cycle $C \in E(\cC_{v_0})$, we set
\begin{equation*}
    r(\sigma, C) \defined
    \begin{cases}
        \card{C \cap \rU_\sigma} & \text{if $C \subseteq \rU_\sigma \cup \rM_\sigma$}, \\
        \infty & \text{otherwise,}
    \end{cases}
\end{equation*}
and for a fixed $0 < \lambda < 1$, to be chosen later, we define the \indef{danger} of $C$ as
\begin{equation*}
    w_\lambda(\sigma, C) \defined  \lambda^{r(\sigma, C)}.
\end{equation*}
The total danger associated with a configuration $\sigma$ is then defined to be
\begin{equation*}
    w_\lambda(\sigma, \cC_{v_0}) \defined \sum_{C \in E(\cC_{v_0})} w_\lambda(\sigma, C).
\end{equation*}
To proceed as before, we need an analogue of \Cref{prop:weight-saw} for dual cycles, that is, to determine conditions that guarantee that $w_\lambda(\sigma^{\rR}, \cC_{v_0})$ is controlled.
The trouble is that the total danger now is an infinite sum, as $\card{E(\cC_{v_0})} = \infty$.
With the right choice of parameters, the total danger turns into a convergent sum.
However, we need the total danger of the initial configuration to be low, rather than just finite, in order to use \Cref{thm:beck-infinite}.
As an appropriate tail of a convergent sum is arbitrarily small, the initial danger can be made low if Maker could ignore short dual cycles.
To turn this observation into a strategy for Maker, she split the board $\ZZ^2$ into two disjoint regions, one finite and one infinite.
From Maker's point of view, there are now two games being played in parallel, one for each region.
Whenever Breaker claims an edge in a given region, she responds by playing a move in that same region, following her corresponding strategy.
We provide Maker with a strategy in the finite region that guarantees that no short dual cycle of the board will be fully claimed by Breaker.
On the infinite region, Maker plays the potential-based strategy described above.

Denote by $L_{v_0}^N$ the subgraph of $\ZZ^2$ induced by the vertices $v_0 + \set{0, 1} \times \set {0, 1, \dotsc, N}$.
The \indef{protective game} on $L_{v_0}^N$ is played on the edges $E(L_{v_0}^N)$ and following the same dynamics of the Maker-Breaker percolation game.
Maker's goal is to guarantee that Breaker never claims a full cut in $L_{v_0}^N$.
In other words, the removal of the edges claimed by Breaker cannot disconnect $L_{v_0}^N$. 
As we will now see, Maker can accomplish this task quite easily, regardless of the value of $N$, as long as $b = 1$, $m \geq 2$, and all the edges in $L_{v_0}^N$ are unclaimed in the initial configuration.
Maker's strategy for the protective game is as follows.
In the first round of the game, Maker claims the edge $\set{v_0, v_0 + (0,1)}$.
After that, Maker partitions the remaining edges $E(L_{v_0}^N)$ into sets of size three as in \Cref{fig:ladder}.
Indeed, the parts are of the form
\begin{equation*}
    \set[\big]{\set{v,v+(1,0)}, \set{v+(1,0),v+(1,1)}, \set{v+(1,0),v+(1,1)} }
\end{equation*}
for $v \in \set{v_0, v_0 + (0,1), \dotsc v_0 + (0,N-1)}$.
Maker plays now a reactive strategy; whenever Breaker claims an unclaimed edge in $E(L_{v_0}^N)$, Maker claims the other two edges in the same part of the aforementioned partition.
To see that this is indeed a winning strategy for the protective game, notice that every cut in $L_{v_0}^N$ without the edge $\set{v_0, v_0 + (0,1)}$ has at least two edges in one of the triples above.

\begin{figure}[ht!]
\centering
\begin{tikzpicture}[scale=1.0]
\tikzstyle{s-helper}=[black, line width=0.1em];
\definecolor{s-col}{RGB}{135, 201, 180};
\tikzstyle{s-line}=[line width=0.4em, s-col, draw opacity=1];
\def\d{0.2}
\draw[s-line]
(0+\d,0)--(1,0)--(1,1)--(0+\d,1)
(1+\d,0)--(2,0)--(2,1)--(1+\d,1)
(2+\d,0)--(3,0)--(3,1)--(2+\d,1)
(3+\d,0)--(4,0)--(4,1)--(3+\d,1)
(4+\d,0)--(5,0)--(5,1)--(4+\d,1)
(5+\d,0)--(6,0)--(6,1)--(5+\d,1)
(6+\d,0)--(7,0)--(7,1)--(6+\d,1);
\draw[s-helper]
(0,0)--(7,0) (0,1)--(7,1)
(0,0)--(0,1) (1,0)--(1,1) (2,0)--(2,1)
(3,0)--(3,1) (4,0)--(4,1) (5,0)--(5,1)
(6,0)--(6,1) (7,0)--(7,1);
\foreach \x / \y in
{0/0,0/1,1/0,1/1,2/0,2/1,3/0,3/1,4/0,4/1,5/0,5/1,6/0,6/1,7/0,7/1}
\filldraw[xshift=\x cm, yshift=\y cm, fill=black] (0,0) circle (3pt);
\draw (-0.5,0) node {$v_0$};
\end{tikzpicture}
\caption{Disjoint triples of edges in $L_{v_0}^N$ used in Maker's strategy.}
\label{fig:ladder}
\end{figure}
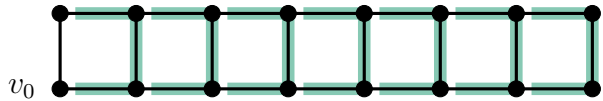

An alternative way to see that the reactive strategy is a winning one is by considering the dual of the edges claimed by Breaker.
Note that Maker guarantees that any dual path of edges claimed by Breaker that enters from the boundary of $L_{v_0}^N$ can only be continued `to the left'.
The fact that Maker claimed the leftmost edge $\set{v_0, v_0 + (0,1)}$ impedes any dual path of Breaker from leaving $L_{v_0}^N$, so Breaker cannot claim a minimal cut.
See \Cref{fig:ladder-claimed} for an example.

\begin{figure}[ht!]
\centering
\begin{tikzpicture}[scale=1.0]
\tikzstyle{s-maker}=[blue, ultra thick];
\tikzstyle{s-breaker}=[red, thick];
\tikzstyle{s-breaker-dual}=[red!20, ultra thick];
\draw[s-breaker-dual]
(0.5,0.5) -- (1.5,0.5) -- (1.5,-0.5)
(2.5,0.5) -- (4.5,0.5) -- (4.5,1.5)
(5.5,0.5) -- (7.5,0.5)
;

\draw[s-maker]
(1,0) -- (0,0) -- (0,1) -- (2,1) -- (2,0) -- (7,0)
(2,1) -- (4,1)
(5,0) -- (5,1) -- (7,1)
;

\draw[s-breaker]
(2,0) -- (1,0) -- (1,1)
(3,0) -- (3,1)
(4,0) -- (4,1) -- (5,1)
(6,0) -- (6,1)
(7,0) -- (7,1)
;

\foreach \x / \y in
{0.5/0.5,1.5/0.5,2.5/0.5,3.5/0.5,4.5/0.5,5.5/0.5,6.5/0.5,7.5/0.5,1.5/-0.5,4.5/1.5}
\filldraw[xshift=\x cm, yshift=\y cm, fill=red, red!60, thick] (-0.1,-0.1) -- (-0.1,0.1) -- (0.1,0.1) -- (0.1,-0.1) -- cycle;

\foreach \x / \y in
{0/0,0/1,1/0,1/1,2/0,2/1,3/0,3/1,4/0,4/1,5/0,5/1,6/0,6/1,7/0,7/1}
\filldraw[xshift=\x cm, yshift=\y cm, fill=black] (0,0) circle (3pt);
\draw (-0.5,0) node {$v_0$};
\end{tikzpicture}
\caption{Example of fully claimed $L_{v_0}^N$ in the protective game. Dual paths claimed by Breaker shown in red.}
\label{fig:ladder-claimed}
\end{figure}
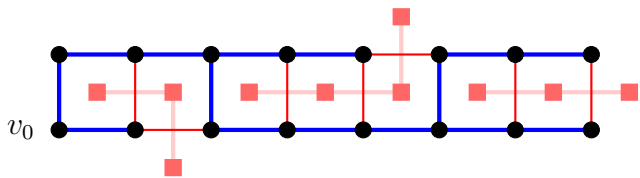

In view of this reactive strategy, we may modify the \Cref{obs:maker-hypergraph} with the new information that no dual cycle of length $\leq N$ is ever claimed by Breaker, given that Maker plays the protective game on $L_{v_0}^N$.
We want to equip Maker with the potential-based strategy on the board $\ZZ^2$ restricted to the complement of $E(L_{v_0}^N)$.
Indeed, define the hypergraph $\cC_{v_0}^N$ to be the subhypergraph of $\cC_{v_0}$ induced by $E(\ZZ^2) \setminus E(L_{v_0}^N)$.
In other words, $\cC_{v_0}^N$ is obtained from $\cC_{v_0}$ by erasing $E(L_{v_0}^N)$, together with all dual cycles that use at least one of those edges.
Therefore, the hyperedges of $\cC_{v_0}^N$ are all dual cycles in $\ZZ^2$ that fully encircle $L_{v_0}^N$, and therefore, have length at least $N$.
We update \Cref{obs:maker-hypergraph} in the following way.

\begin{observation}
\label{obs:maker-hypergraph-refined}
Let $m \geq 2$ and $b = 1$.
Suppose that $\sigma$ is a game configuration on $\ZZ^2$ with the property that there is $v_0 \in \ZZ^2$ and $N \geq 0$ such that:
\begin{enumerate}
    \item $E(L_{v_0}^N)$ has no edges claimed by Breaker in $\sigma$;
    \item Breaker has a strategy for the $(b,m)$ Maker-Breaker game on $\cC_{v_0}^N$ with initial configuration $\sigma$, restricted to $\cC_{v_0}^N$, in which Maker never fully claims a hyperedge.
\end{enumerate}
Then Maker has a winning strategy for the $(m,b)$ Maker-Breaker percolation game on $\ZZ^2$ with initial configuration $\sigma^{\rR}$.
\end{observation}
\begin{proof}
Maker chooses $v_0$ to protect.
As having edges already claimed by Maker can only help her, we may assume that all edges in $E(L_{v_0}^N)$ are unclaimed.
Maker follows a reactive strategy that plays two games at once: she plays as Maker in the protective game on $L_{v_0}^N$ and she plays as Breaker in the $(b,m)$ Maker-Breaker game on $\cC_{v_0}^N$ .
On her first turn, Maker follows her strategy in the protective game on $L_{v_0}^N$ and claims the edge $\set{v_0, v_0 + (0,1)}$.

After her first turn, she will now respond by playing either in $E(L_{v_0}^N)$ or in $E(\ZZ^2) \setminus E(L_{v_0}^N)$ depending on where Breaker makes a move.
As $b = 1$, Breaker can only make a move in one of these two games in each turn.
If Breaker claims an edge in $E(L_{v_0}^N)$, then Maker responds according to her winning strategy for the protective game on $L_{v_0}^N$.
If Breaker claims an edge in $E(\ZZ^2) \setminus E(L_{v_0}^N)$, then Maker plays as Breaker would in the $(b,m)$ Maker-Breaker game on $\cC_{v_0}^N$.

We claim that Maker wins the $(m,b)$ Maker-Breaker percolation game on $\ZZ^2$ by following this composite strategy.
Assume, for contradiction, that Breaker wins by fully claiming the dual cycle $C$ surrounding $v_0$.
If no edge in $C$ is contained in $E(L_{v_0}^N)$ then $C$ is a hyperedge in $\cC_{v_0}^N$, contradicting the fact that Breaker (playing as Maker) never claims a full hyperedge in the game on $\cC_{v_0}^N$.
If $C$ contains an edges from $E(L_{v_0}^N)$ then, since it is a cycle, it intersects $L_{v_0}^N$ in a cut, contradicting the fact that Maker wins that protective game on $L_{v_0}^N$.
\end{proof}

Now that we have a promising setup, we analyse when a initial configuration $\sigma$ coming from $(\ZZ^2)_{\alpha, \beta}$ will allow the existence of appropriate $v_0$ and $N$ such that \Cref{obs:maker-hypergraph-refined} can be applied.
While the condition that $L_{v_0}^N$ is unclaimed in $\sigma$ is a minor requirement, to apply \Cref{thm:beck-infinite}, we must guarantee that $w_\lambda(\sigma^{\rR}, \cC_{v_0}^N)$ is quite small.

\begin{proposition}
\label{prop:weight-dual-cycle}
Let $\kappa$ be the connective constant of $\ZZ^2$, let $\alpha, \beta \geq 0$, $\alpha + \beta \leq 1$ and let $0 < \lambda < 1$ be such that $\p[\big]{\lambda(1 - \alpha - \beta) + \beta} \kappa < 1$.
Then almost surely, a random game configuration $\sigma$ distributed as $(\ZZ^2)_{\alpha,\beta}$ has the property that for every $\delta > 0$ there exists $v_0 \in \ZZ^2$ such that:
\begin{enumerate}
    \item $ w_\lambda(\sigma^{\rR}, \cC_{v_0}^N) < \delta$;
    \item $E(L_{v_0}^N)$ is fully unclaimed in $\sigma$.
\end{enumerate}
\end{proposition}
\begin{proof}
Let $\sigma$ be a random game configuration distributed as $(\ZZ^2)_{\alpha,\beta}$ and let $v \in \ZZ^2$ be any vertex.
Every self-avoiding dual cycle $C \in \cC_v^N$ with $\card{C} = n$ must contain at least one of the edges in $\set{ (x, x+1) \st 0 \leq x \leq n }$.
Therefore, we can bound the number of $C \in \cC$ with $\card{C} = n$ by $n \card{E(\cP_{\ZZ^2, v_0}^n)}$.
As in the proof of \Cref{prop:weight-saw}, there exists $n_0(\eps)$ such that there are at most $(\kappa + \eps)^n$ such dual cycles if $n \geq n_0(\eps)$.

Note that $w_\lambda(\sigma^{\rR}, C) = \ind_{\set{C \cap \rM_\sigma = \emptyset}} \lambda^{\card{C \cap \rU_\sigma}}$.
Given $C \in \cC$ with $\card{C} = n$, we have
\begin{align*}
    \expec[\big]{w_\lambda(\sigma^{\rR}, C)}
    &= \sum_{0 \leq k \leq n} \lambda^k \pbond[\big]{\ZZ^2,\alpha,\beta}{\card{C \cap \rU_\sigma} = k, \card{C \cap \rM_\sigma} = 0} \\
    &= \sum_{0 \leq k \leq n} \lambda^k \binom{n}{k}(1 - \alpha - \beta)^k \beta^{n-k}
    = \p[\big]{\lambda(1 - \alpha - \beta) + \beta}^n.
\end{align*}
Thus, for $N \geq n_0(\eps)$, we have
\begin{align*}
    \expec[\big]{w_\lambda(\sigma^{\rR}, \cC_v^N)}
    &= \sum_{n \geq N} \sum_{\substack{C \in E(\cC_v^N) \\ \card{C} = n }} \expec[\big]{w_\lambda(\sigma, P)} \\
    &\leq \sum_{n \geq N} \p[\Big]{(\kappa + \eps)\p[\big]{\lambda (1 - \alpha - \beta) + \beta}}^n.
\end{align*}
We can choose $\eps$ small enough such that $(\kappa + \eps)\p[\big]{\lambda (1 - \alpha - \beta) + \beta} < 1$ so that the series above converges, and $N$ large enough such that the tail of this series is arbitrarily small.
By Markov's inequality, if $N$ is large enough, we have
\begin{equation*}
    \pbond{\ZZ^2,\alpha,\beta}{w_\lambda(\sigma^\rR, \cC_v^N) \geq \delta}
    \leq \delta^{-1}\expec[\big]{w_\lambda(\sigma^{\rR}, \cC_v^N)} \leq 1/2.
\end{equation*}

Note that the events `$w_\lambda(\sigma^\rR, \cC_v^N) < \delta$' and `$\card{\rB_\sigma \cap E(L_v^N)} = 0$' are independent, as they are fully determined by disjoint sets of edges.
Let $\cE_{v,N}$ be the union of these two events.
For $N$ sufficiently large, we have
\begin{align*}
    \pbond[\big]{\ZZ^2,\alpha,\beta}{\cE_{v,N}}
    &= \pbond[\big]{\ZZ^2,\alpha,\beta}{w_\lambda(\sigma^\rR, \cC_v^N) \leq \delta} \pbond[\big]{\ZZ^2,\alpha,\beta}{\card{\rB_\sigma \cap E(L_v^N)} = 0} \\
    &\geq \frac{1}{2} (1 - \beta)^{3N+1} > 0.
\end{align*}
Denote by $\cE_{N} = \bigcup_{v \in \ZZ^2} \cE_{v,N}$.
Since $\cE_{N}$ is a tail event, the Kolmogorov’s zero--one law gives us that $\pbond{\ZZ^2,\alpha,\beta}{\cE_N} \in \set{0,1}$.
As $\pbond{\ZZ^2,\alpha,\beta}{\cE_N} \geq \pbond{\ZZ^2,\alpha,\beta}{\cE_{(0,0),N}} > 0$ for $N \geq N_0$, we have $\pbond{\ZZ^2,\alpha,\beta}{\cE_N} = 1$ for $N \geq N_0$, completing the proof.
\end{proof}

Assuming \Cref{thm:beck-infinite}, we obtain \Cref{thm:potential-maker-symmetric} from \Cref{prop:weight-dual-cycle}.

\begin{proof}[Proof of \Cref{thm:potential-maker-symmetric}]
Setting $\lambda = (m+1)^{-1}$ in the condition $\p[\big]{\lambda(1 - \alpha - \beta) + \beta} \kappa < 1$, we obtain \eqref{eq:mb-maker-ab}.
Let $\sigma_0$ be a random game configuration distributed as $(\Lambda)_{\alpha, \beta}$.
By \Cref{prop:weight-saw} we know that almost surely there exist $N \geq 0$ and $v_0 \in \ZZ^2$ such that $w_\lambda(\sigma, \cC_{v_0}^N) < 1/(m+1)$ and $L_{v_0}^N$ has no edge claimed by Breaker.
Under this condition, by \Cref{thm:beck-infinite}, Breaker has a strategy for the $(b,m)$ Maker-Breaker game on $\cC_{v_0}^N$ with initial configuration $\sigma^{\rR}$ in which Maker never fully claims a hyperedge.
By \Cref{obs:maker-hypergraph-refined}, we have that Breaker almost surely has a winning strategy for the $(m, b)$ Maker-Breaker percolation game on $(\Lambda)_{\alpha, \beta}$, as claimed.
\end{proof}


\section{Analysis of the potential-based strategies}
\label{sec:analysis}

The goal of this section is to provide an analysis of the potential-based strategy in a Maker-Breaker game, leading to a proof of \Cref{thm:beck-infinite}, which is equivalent to \Cref{thm:beck-infinite-alternative} below.
This completes the proof of \Cref{thm:potential-breaker-symmetric} and \Cref{thm:potential-maker-symmetric}, together with their many consequences.
We also extend the analysis to boosted games in \Cref{cor:beck-infinite-boosted}.

Fix a possibly infinite hypergraph $\cH$.
In this section, we describe and analyse a so-called \indef{potential-based strategy} for Breaker in the $(m,b)$ Maker-Breaker game on $\cH$.
Recall that a \indef{game configuration} on $\cH$ is a map $\sigma \from V(\cH) \to \set{\rU, \rM, \rB}$ and that we write $\rM_\sigma \defined \sigma^{-1}(\rM)$, $\rB_\sigma \defined \sigma^{-1}(\rB)$ and $\rU_\sigma \defined \sigma^{-1}(\rU)$.
Given a game configuration $\sigma$ and an unclaimed vertex $v \in \rU_\sigma$, denote by $\sigma_{\rM \gets v}$ and $\sigma_{\rB \gets v}$ the game configurations satisfying
\begin{align*}
    \sigma_{\rM \gets v}(u) \defined \begin{cases}
        \rM & \text{if $u = v$}, \\
        \sigma(u) & \text{otherwise},
    \end{cases}
    & &
    \sigma_{\rB \gets v}(u) \defined \begin{cases}
        \rB & \text{if $u = v$}, \\
        \sigma(u) & \text{otherwise}.
    \end{cases}
\end{align*}

To win the $(m,b)$ Maker-Breaker percolation game on $\cH$, Breaker has to claim one vertex in every hyperedge of $\cH$, recall \Cref{sec:hyper}.
We provide Breaker with a way of evaluating how dangerous an hyperedge $S \in E(\cH)$ is given a game configuration $\sigma$.
Breaker will then play greedily, always choosing vertices that minimise the estimated danger.
To be more precise, define
\begin{equation*}
    r(\sigma, S) \defined
    \begin{cases}
        \card{S \cap \rU_\sigma} & \text{if $S \subseteq \rU_\sigma \cup \rM_\sigma$}, \\
        \infty & \text{otherwise,}
    \end{cases}
\end{equation*}
and for a fixed $0 < \lambda < 1$ to be chosen later, define the \indef{danger} of an hyperedge $S$ as
\begin{equation*}
    w_\lambda(\sigma, S) \defined  \lambda^{r(\sigma, S)}.
\end{equation*}
Indeed, the energy $w_\lambda(\sigma, S)$ is a reasonable estimate for the danger of an hyperedge $S$ from Breaker's point of view.
As $w_\lambda(\sigma, S) = 0$ whenever Breaker already claimed a vertex in $S$, he can cease to worry about $S$.
On the other hand, if the vertices of $S$ are either unclaimed or claimed by Maker, $w_\lambda(\sigma, S)$ is positive.
As the number of vertices left unclaimed in $S$ decreases, the danger associated with such an hyperedge grows exponentially.
If $w_\lambda(\sigma, S) = 1$, Breaker failed his strategy and Maker fully claimed the hyperedge $S$.

The total danger associated with a configuration $\sigma$ is defined as
\begin{equation*}
    w_\lambda(\sigma) \defined \sum_{S \in E(\cH)} w_\lambda(\sigma, S).
\end{equation*}
The danger $w_\lambda(\sigma)$ can be interpreted as an \indef{energy} that Breaker aims to minimise.

With the potential-based strategy, Breaker always play a vertex that minimises $w_\lambda$.
His goal is to guarantee that Maker will never fully claim a winning set $S \in E(\cH)$.
In \Cref{thm:beck-infinite-alternative}, we show that by playing in this way, $w_\lambda(\sigma_t) \leq w_\lambda(\sigma_0)$ for all $t$.
Therefore, if $w_\lambda(\sigma_0) < 1$, than the same holds for all $\sigma_t$.
In particular, Breaker attains his aforementioned goal.

To analyse the energetic strategy, we need to keep track of some key quantities.
Define the \emph{$\rB$-potential} of a vertex $v \in \rU_\sigma$ to be
\begin{align*}
    \Delta_{\rB \gets v}(\sigma) \defined w_\lambda(\sigma) - w_\lambda(\sigma_{\rB \gets v}),
\end{align*}
and its \emph{$\rM$-potential} to be
\begin{align*}
    \Delta_{\rM \gets v}(\sigma) \defined w_\lambda(\sigma_{\rM \gets v}) - w_\lambda(\sigma).
\end{align*}
In other words, if Breaker claims $v$, the danger is reduced by $\Delta_{\rB \gets v}(\sigma)$, whereas Maker increases the danger by $\Delta_{\rM \gets v}(\sigma)$ by claiming $v$.

Before the proof of the main result of the section, \Cref{thm:beck-infinite-alternative}, we collect some observations.
For future reference, notice that $r(\sigma, S)$ changes in the following way when the players claim an unclaimed vertex $v$:
\begin{equation}
\label{eq:risk-breaker}
    r(\sigma_{\rB \gets v}, S) \defined
    \begin{cases}
        r(\sigma, S) & \text{if $v \notin S$}, \\
        \infty & \text{if $v \in S$},
    \end{cases}
\end{equation}
and
\begin{equation}
\label{eq:risk-maker}
    r(\sigma_{\rM \gets v}, S) \defined
    \begin{cases}
        r(\sigma, S) & \text{if $v \notin S$}, \\
        r(\sigma, S) - 1 & \text{if $v \in S$}.
    \end{cases}
\end{equation}
We collect some further observations.

\begin{observation}
\label{obs:delta-breaker}
For any unclaimed vertex $v \in \rU_\sigma$, we have
\begin{equation*}
    \Delta_{\rB \gets v}(\sigma) = \sum_{v \in S \in E(\cH)} \lambda^{r(\sigma, S)}.
\end{equation*}
\end{observation}
\begin{proof}
From \eqref{eq:risk-breaker}, we get
\begin{align*}
    \Delta_{\rB \gets v}(\sigma)
    &= w_\lambda(\sigma) - w_\lambda(\sigma_{\rB \gets v}) \\
    &= \sum_{S \in E(\cH)} \p[\Big]{ w_\lambda(\sigma, S) - w_\lambda(\sigma_{\rB \gets v}, S)}
    = \sum_{v \in S \in E(\cH)} \lambda^{r(\sigma, S)},
\end{align*}
as claimed.
\end{proof}

\begin{observation}
\label{obs:delta-maker}
For any unclaimed vertex $v \in \rU_\sigma$, we have
\begin{align*}
    \Delta_{\rM \gets v}(\sigma) = (\lambda^{-1} - 1)\Delta_{\rB \gets v}(\sigma).
\end{align*}
\end{observation}
\begin{proof}
From \eqref{eq:risk-maker} and \Cref{obs:delta-breaker}, we get
\begin{align*}
    \Delta_{\rM \gets v}(\sigma)
    &= w_\lambda(\sigma_{\rM \gets v}) - w_\lambda(\sigma) \\
    &= \sum_{S \in E(\cH)} \p[\Big]{ w_\lambda(\sigma_{\rM \gets v}, S) - w_\lambda(\sigma, S)} \\
    &= \sum_{v \in S \in E(\cH)} \p[\Big]{ \lambda^{r(\sigma, S) - 1} - \lambda^{r(\sigma, S)}}
    = (\lambda^{-1} - 1) \sum_{v \in S \in E(\cH)} \lambda^{r(\sigma, S)}.
\end{align*}
Therefore $\Delta_{\rM \gets v}(\sigma) = (\lambda^{-1} - 1) \Delta_{\rB \gets v}(\sigma)$ as claimed.
\end{proof}

\begin{observation}
\label{obs:monotone}
For any pair of unclaimed vertices $v_1, v_2 \in \rU_\sigma$, we have
\begin{equation*}
    \Delta_{\rB \gets v_2}(\sigma_{\rB \gets v_1}) \leq \Delta_{\rB \gets v_2}(\sigma).
\end{equation*}
\end{observation}
\begin{proof}
From \Cref{obs:delta-breaker} and \eqref{eq:risk-breaker}, we get
\begin{align*}
    \Delta_{\rB \gets v_2}(\sigma_{\rB \gets v_1})
    &= \sum_{v_2 \in S \in E(\cH)} \lambda^{r(\sigma_{\rB \gets v_1}, S)}
    = \sum_{\substack{v_2 \in S \in E(\cH) \\ v_1 \notin S}} \lambda^{r(\sigma, S)} \\
    &\leq \sum_{v_2 \in S \in E(\cH)} \lambda^{r(\sigma, S)}
    = \Delta_{\rB \gets v_2}(\sigma),
\end{align*}
as claimed.
\end{proof}

\begin{observation}
\label{obs:order}
For any pair of unclaimed vertices $v_1, v_2 \in \rU_\sigma$, either it holds that
\begin{equation*}
    \Delta_{\rM \gets v_2}(\sigma_{\rM \gets v_1}) \leq \lambda^{-1}\Delta_{\rM \gets v_1}(\sigma),
\end{equation*}
or the same holds with $v_1$ and $v_2$ exchanged.
\end{observation}
\begin{proof}
Assume, without loss of generality, that
\begin{align*}
    \sum_{\substack{v_2 \in S \in E(\cH) \\ v_1 \notin S}} \lambda^{r(\sigma, S)}
    \leq \sum_{\substack{v_1 \in S \in E(\cH) \\ v_2 \notin S}} \lambda^{r(\sigma, S)}.
\end{align*}
From \Cref{obs:delta-breaker}, we have
\begin{align*}
    \Delta_{\rB \gets v_2}(\sigma_{\rM \gets v_1})
    &= \sum_{v_2 \in S \in E(\cH)} \lambda^{r(\sigma_{\rM \gets v_1}, S)} \\
    &= \sum_{v_1, v_2 \in S \in E(\cH)} \lambda^{r(\sigma, S) - 1}
    + \sum_{\substack{v_2 \in S \in E(\cH) \\ v_1 \notin S}} \lambda^{r(\sigma, S)} \\
    &\leq  \lambda^{-1} \sum_{v_1, v_2 \in S \in E(\cH)} \lambda^{r(\sigma, S)} + \sum_{\substack{v_1 \in S \in E(\cH) \\ v_2 \notin S}} \lambda^{r(\sigma, S)} \\
    &\leq \lambda^{-1} \sum_{v_1 \in S \in E(\cH)} \lambda^{r(\sigma, S)}
    = \lambda^{-1} \Delta_{\rB \gets v_1}(\sigma).
\end{align*}
Finally, from \Cref{obs:delta-maker}, we have that $\Delta_{\rM \gets v_2}(\sigma_{\rM \gets v_1}) \leq \lambda^{-1}\Delta_{\rM \gets v_1}(\sigma)$.
\end{proof}

\begin{observation}
\label{obs:maker-increment}
For any unclaimed vertex $v \in \rU_\sigma$, we have
\begin{align*}
    w_\lambda(\sigma_{\rM \gets v}) \leq \lambda^{-1}w_\lambda(\sigma).
\end{align*}
\end{observation}
\begin{proof}
From \Cref{obs:delta-maker} and since $w_\lambda(\sigma) = w_\lambda(\sigma_{\rB \gets v}) + \Delta_{\rB \gets v}(\sigma)$, we get
\begin{align*}
    w_\lambda(\sigma_{\rM \gets v})
    &= w_\lambda(\sigma) + \Delta_{\rM \gets v}(\sigma) \\
    &= w_\lambda(\sigma) + (\lambda^{-1} - 1)\Delta_{\rB \gets v}(\sigma) \\
    &= w_\lambda(\sigma_{\rB \gets v}) + \lambda^{-1}\Delta_{\rB \gets v}(\sigma) \\
    &\leq \lambda^{-1}\p[\big]{w_\lambda(\sigma_{\rB \gets v}) + \Delta_{\rB \gets v}(\sigma)}
    = \lambda^{-1}w_\lambda(\sigma). \qedhere
\end{align*}
\end{proof}

We are now ready to analyse that the potential-based strategy and give a proof of \Cref{thm:beck-infinite}, which is stated below with a slight change of notation.

\begin{theorem}
\label{thm:beck-infinite-alternative}
Let $\cH$ be a hypergraph and let $m,b \geq 1$ be integers.
Let $\sigma_0$ be a game configuration on $\cH$.
If
\begin{equation*}
    w_\lambda(\sigma_0) = \sum_{S \in E(\cH)} w_\lambda(\sigma_0, S) < \frac{1}{b+1},
\end{equation*}
for $\lambda = (b+1)^{-1/m}$.
Then Breaker has a strategy for the $(m,b)$ Maker-Breaker game on $\cH$ with initial configuration $\sigma_0$ which guarantees that Maker never fully claims a winning set $S \in E(\cH)$.
In particular, Breaker wins if $\cH$ is finite.
\end{theorem}
\begin{proof}
Write $g = m + b$.
The game starts with configuration $\sigma_0$.
We write $v(i)$ for the $i$-th claimed vertex from round $1$ onward.
The game configurations on the $(k+1)$-st round are the following ones
\begin{align*}
    \sigma_{gk + 1} &= (\sigma_{gk})_{\rM \gets v(gk + 1)}, &
    \sigma_{gk + m + 1} &= (\sigma_{gk + m})_{\rB \gets v(gk + m + 1)}, \\
    \sigma_{gk + 2} &= (\sigma_{gk + 1})_{\rM \gets v(gk + 2)}, &
    \sigma_{gk + m + 2} &= (\sigma_{gk + m + 1})_{\rB \gets v(gk + m + 2)}, \\
    &\;\;\vdots & &\;\;\vdots \\
    \sigma_{gk + m} &= (\sigma_{gk + m - 1})_{\rM \gets v(gk + m)}, &
    \sigma_{gk + m + b} &= (\sigma_{gk + m + b - 1})_{\rB \gets v(gk + m + b)}.
\end{align*}
For any $k \geq 1$ we have $gk-b = g(k-1)+m$ and the corresponding game configuration $\sigma_{gk-b} = \sigma_{g(k-1)+m}$ is the one obtained at the end of Maker's $(k-1)$-st turn.

We equip Breaker with the following potential-based strategy: Breaker will always claim an unclaimed vertex $v \in \rU_\sigma$ that maximises the potential $\Delta_{\rB \gets v}(\sigma)$.
In other words, for every $k \geq 0$ and $1 \leq i \leq b$, and for a unclaimed vertex $v \in \rU_{\sigma_{gk-i}}$ we have
\begin{equation}
\label{eq:breaker-max}
  \Delta_{\rB \gets v}(\sigma_{gk-i}) \leq \Delta_{\rB \gets v(gk-i+1)}(\sigma_{gk-i}).
\end{equation}
This implies, for instance, that
\begin{align*}
    \Delta_{\rB \gets v(gk)}(\sigma_{gk-1})
    &= \Delta_{\rB \gets v(gk)}((\sigma_{gk-2})_{\rB \gets v(gk-1)}) & & \\
    &\leq \Delta_{\rB \gets v(gk)}(\sigma_{gk-2}) & & (\text{by \Cref{obs:monotone}}) \\
    &\leq \Delta_{\rB \gets v(gk-1)}(\sigma_{gk-2}) & & (\text{by \eqref{eq:breaker-max}}).
\end{align*}
By induction, we get that
\begin{align}
\label{eq:breaker-mono}
    \Delta_{\rB \gets v(gk)}(\sigma_{gk-1})
    &\leq \Delta_{\rB \gets v(gk-1)}(\sigma_{gk-2})
    \leq \dotsb \leq \Delta_{\rB \gets v(gk-b+1)}(\sigma_{gk-b}).
\end{align}

Breaker objective is to prevent Maker to fully claim a winning set $S \in E(\cH)$.
This is the case if $w_\lambda(\sigma_t) < 1$ for all $t \geq 1$.
To show this, it suffices to have $w_\lambda(\sigma_m) < 1$ and
\begin{equation}
\label{eq:breaker-goal}
    w_\lambda(\sigma_{gk+m}) \leq w_\lambda(\sigma_{gk-b}), \quad \text{ for all $k \geq 1$.}
\end{equation}
The condition that $w_\lambda(\sigma_m) < 1$ is satisfied as $w_\lambda(\sigma_0) < 1/(b+1)$.
Indeed, by \Cref{obs:maker-increment}, we have $w_\lambda(\sigma_m) \leq \lambda^{-m} w_\lambda(\sigma_0) = (b+1) w_\lambda(\sigma_0) < 1$.
We will find an equivalent formulation of \eqref{eq:breaker-goal} that is easier to handle.
Observe that
\begin{align*}
    w_\lambda(\sigma_{gk-b}) - w_\lambda(\sigma_{bk+m})
    &= \sum_{\ell=1}^g \p[\big]{w_\lambda(\sigma_{gk-b+\ell-1}) - w_\lambda(\sigma_{gk-b+\ell}) } \\
    &= \sum_{i=1}^b \Delta_{\rB \gets v(gk-b+i)}(\sigma_{gk-b+i-1}) - \sum_{j=1}^m \Delta_{\rM \gets v(gk+j)}(\sigma_{gk+j-1}),
\end{align*}
as we have, for $1 \leq i \leq b$,
\begin{align*}
    w_\lambda(\sigma_{gk-b+i-1}) - w_\lambda(\sigma_{gk-b+i})
    &= w_\lambda(\sigma_{gk-b+i-1}) - w_\lambda((\sigma_{gk-b+i-1})_{\rB \gets v(gk-b+i)}) \\
    &= \Delta_{\rB \gets v_{gk-b+i}}(\sigma_{gk-b+i-1}),
\end{align*}
and similarly for $1 \leq j \leq m$, we have
\begin{align*}
    w_\lambda(\sigma_{gk+j-1}) - w_\lambda(\sigma_{gk+j})
    &= w_\lambda(\sigma_{gk+j-1}) - w_\lambda((\sigma_{gk+j-1})_{\rM \gets v(gk+j)}) \\
    &= - \Delta_{\rM \gets v(gk+j)}(\sigma_{gk+j-1}).
\end{align*}
Reordering the sum in $i$, \eqref{eq:breaker-goal} is equivalent to
\begin{equation}
\label{eq:breaker-goal-mb}
    \Delta_{\rM} \defined \sum_{j=1}^m \Delta_{\rM \gets v(gk+j)}(\sigma_{gk+j-1})
    \leq \sum_{i=1}^b \Delta_{\rB \gets v(gk-i+1)}(\sigma_{gk-i}) \defines \Delta_{\rB}.
\end{equation}

Observe that \eqref{eq:breaker-mono} gives us
\begin{align}
\label{eq:breaker-lower-bound}
    \Delta_{\rB} = \sum_{i=1}^b \Delta_{\rB \gets v(gk-i+1)}(\sigma_{gk-i}) \geq b \Delta_{\rB \gets v(gk)}(\sigma_{gk-1})
\end{align}
Repeatedly applying \Cref{obs:order}, and possibly reordering the vertices $v(gk+j)$, $1 \leq j \leq m$, we have
\begin{align*}
    \Delta_{\rM} = \sum_{j = 1}^{m}\Delta_{\rM \gets v(gk+j)}(\sigma_{gk+j-1}) \leq (1 + \lambda^{-1} + \dotsb + \lambda^{-m+1}) \Delta_{\rM \gets v(gk+1)}(\sigma_{gk}).
\end{align*}
Now, by \Cref{obs:delta-maker}, we have
\begin{align*}
    \Delta_{\rM} &\leq (1 + \lambda^{-1} + \dotsb + \lambda^{-m+1}) (\lambda^{-1} - 1) \Delta_{\rB \gets v(gk+1)}(\sigma_{gk})\\
    &= (\lambda^{-m} - 1) \Delta_{\rB \gets v(gk+1)}(\sigma_{gk}).
\end{align*}
But observe that
\begin{align*}
    \Delta_{\rB \gets v(gk+1)}(\sigma_{gk})
    &= \Delta_{\rB \gets v(gk+1)}((\sigma_{gk-1})_{\rB \gets v(gk)}) & & (\text{by definition of $\sigma_{gk}$}) \\
    &\leq \Delta_{\rB \gets v(gk+1)}(\sigma_{gk-1}) & & (\text{by \Cref{obs:monotone}}) \\
    &\leq \Delta_{\rB \gets v(gk)}(\sigma_{gk-1}). & & (\text{by \eqref{eq:breaker-max}})
\end{align*}
Finally, this gives
\begin{align*}
    \Delta_{\rM}
    &\leq (\lambda^{-m} - 1) \Delta_{\rB \gets v(gk+1)}(\sigma_{gk}) & & \\
    &\leq (\lambda^{-m} - 1) \Delta_{\rB \gets v(gk)}(\sigma_{gk-1}) & &  \\
    &\leq (\lambda^{-m} - 1) b^{-1} \Delta_{\rB} & & (\text{by \eqref{eq:breaker-lower-bound}}) \\
    &= \Delta_{\rB}. & & (\text{by the choice of $\lambda$})
\end{align*}
Therefore, we have \eqref{eq:breaker-goal-mb} and we are done.
\end{proof}

A minor modification in the proof above handles the case of boosted games.

\begin{corollary}
\label{cor:beck-infinite-boosted}
Let $\cH$ be a hypergraph and let $m,b \geq 1$ be integers.
Let $\sigma_0$ be a game configuration on $\cH$.
If $\lambda = (b+1)^{-1/m}$ and
\begin{equation}
\label{eq:beck-condition-boosted}
    w_\lambda(\sigma_0) < \frac{1}{(b+1)^{(m+c)/m}}.
\end{equation}
Then Breaker has a strategy for the $c$-boosted $(m,b)$ Maker-Breaker game on $\cH$ with initial configuration $\sigma_0$ which guarantees that Maker never fully claims a winning set $S \in E(\cH)$.
In particular, Breaker wins if $\cH$ is finite.
\end{corollary}
\begin{proof}
The proof of \Cref{cor:beck-infinite-boosted} is almost identical to the proof of \Cref{thm:beck-infinite-alternative} with a small modification in the analysis of the first turn of the game.
In fact, if Breaker can guarantee that the total danger is $< 1$ by the end of the first turn, then the result follows.
Now Maker has $m + c$ moves in the first round instead of $m$, so we need to guarantee that $w_\lambda(\sigma_{m+c}) < 1$.
Indeed, by \Cref{obs:maker-increment}, we have $w_\lambda(\sigma_{m+c}) \leq \lambda^{-(m+c)} w_\lambda(\sigma_0) < 1$, as long as $w_\lambda(\sigma_0) < \lambda^{m+c}$, which is given by \eqref{eq:beck-condition-boosted}.
\end{proof}


\section{Open problems and heuristics}
\label{sec:heuristics}

In this final section we review the current state of the art of Maker-Breaker percolation games and pose a few problems and conjectures.
We also review some results and questions due to Day and Falgas-Ravry~\cites{Day2021-ua, Day2021-sa}.

\subsection{Deterministic games on \texorpdfstring{$\ZZ^2$}{Z2}}

We say that $\rho = \rho(\Lambda)$ is a \indef{critical bias ratio} for the board $\Lambda$ if there exists a positive function $\varphi(m) = o(m)$ such that Breaker wins the $(m, \rho m + \varphi(m))$ game on $\Lambda$ and Maker wins the $(m, \rho m - \varphi(m))$ game on $\Lambda$.
Determining whether or not the critical bias ratio exists in a lattice seems to be a hard problem.
In particular, even for the square lattice this is yet unknown, as was originally asked by Day and Falgas-Ravry.

\begin{problem}[Question 5.1 in~\cite{Day2021-ua}]
Determine whether $\rho(\ZZ^2)$ exists.
\end{problem}

As they pointed out, this can be seen as a Maker-Breaker analogue of the Harris-Kesten theorem.
Naturally, the existence and the value of $\rho(\Lambda)$ for other graphs $\Lambda$ are also of great interest.
In their work, Day and Falgas-Ravry also showed the following general result for $\ZZ^2$.

\begin{theorem}[Theorems 1.2, 1.3 and 1.4 in~\cite{Day2021-ua}]
\label{thm:original}
For any $m, b \geq 1$, we have
\begin{enumerate}
    \item Maker has a winning strategy for the $(1,1)$ game on $\ZZ^2$;
    \item If $m \geq 2b$, then Maker has a winning strategy for the $(m,b)$ game on $\ZZ^2$;
    \item If $b \geq 2m$, then Breaker has a winning strategy for the $(m,b)$ game on $\ZZ^2$.
\end{enumerate}
\end{theorem}

In particular, this implies that if $\rho(\ZZ^2)$ exists, then $1/2 \leq \rho(\ZZ^2) \leq 2$.
The constant $2$ in these results is due to the fact that a connected subset of $\ZZ^2$ with $k$ edges has at most $2k + 4$ edges in its edge boundary.
This comprised a barrier for strategies that somehow only deal with local structures of claimed edges.
In our previous work~\cite{Dvorak2021-ad} we surpassed this barrier by showing that Breaker still has a winning strategy if $b \geq (2 - 1/14 + o(1))m$.
Indeed, we established the following result.

\begin{theorem}[Theorem 1.2 in~\cite{Dvorak2021-ad}]
\label{thm:main-old}
Breaker has a winning strategy for the $(m,b)$ Maker-Breaker percolation game on $\ZZ^2$ if $b \geq 2m - \floor{(m - 24)/14}$.
\end{theorem}

Therefore, $1/2 \leq \rho(\ZZ^2) \leq 27/14 \approx 1.93$ should $\rho(\ZZ^2)$ exist, which are the current records.
While \Cref{thm:main-old} breaks the perimetric barrier from Breaker's side, no progress was made over \Cref{thm:original} from Maker's perspective.

\begin{problem}
Show that there exists $\eps > 0$ such that Maker wins the $(m,b)$ game on $\ZZ^2$ with $m \geq (2 - \eps)b$ and $b$ large enough.
\end{problem}

In fact, even the following more modest task is open.

\begin{problem}[Stated in~\cite{Day2021-ua}]
Determine which of the players has a winning strategy for the $(2b-1,b)$ game on $\ZZ^2$.
\end{problem}

We believe the answer should be Maker for large enough $b$.
Perhaps surprisingly, the $(2,2)$ game on $\ZZ^2$ is still open.
If one believes Breaker wins the $(2,2)$ game, then by bias monotonicity it should be easier to prove that he wins the $(2,3)$ game.
Similarly, if Maker wins the $(2,2)$, game then it should be easier to show that she wins the $(3,2)$ game.
All these three games are open.

\begin{problem}[Questions 5.3 and 5.4 in~\cite{Day2021-ua}]
Determine which of the players has a winning strategy for the $(3,2)$, $(2,2)$ or the $(2,3)$ game on $\ZZ^2$.
\end{problem}

These are the smallest pairs $(m,b)$ which are undetermined in $\ZZ^2$.
It will be especially interesting to determine the $(2,2)$ game.

\subsection{Higher dimensions}

As previously alluded in the introduction, not much is known about Maker-Breaker percolation games in higher dimensions.
Since $\ZZ^d$ is a subgraph of $\ZZ^{d+1}$ we have some monotonicity relations: if Maker wins the $(m,b)$ game on $\ZZ^d$, then she also wins the $(m,b)$ game in $\ZZ^{d+1}$, whereas if Breaker wins the $(m,b)$ game on $\ZZ^{d+1}$, then he also wins the $(m,b)$ on $\ZZ^d$.

When $d = 2$, the starting point of our investigation was the pairing strategy for Maker in the $(1,1)$ game.
As shown by Day and Falgas-Ravry, a similar strategy can also be used when $d \geq 3$.

\begin{theorem}[Theorem 1.2 in~\cite{Day2021-ua}]
\label{thm:pairing-highdim}
Maker has a winning strategy for the $(t,t)$ game on $\ZZ^d$, for every $1 \leq t \leq d-1$.
\end{theorem}

In particular, Maker wins the $(1,1)$ and $(2,2)$ games in $\ZZ^3$.
However, we determined in \Cref{cor:deterministic-breaker} that Breaker wins the (boosted) $(1,4)$ game on $\ZZ^3$.

\begin{problem}
Determine which of the players has a winning strategy for the $(1,2)$ and the $(1,3)$ game on $\ZZ^3$.
\end{problem}

It will also be interesting to determine the minimum $b$ such that Breaker wins the $(2,b)$ game on $\ZZ^3$.
We have that $b \leq 22$ from \Cref{cor:deterministic-breaker} together with the fact that $\kappa(\ZZ^3) \leq 4.7387 < \sqrt{21}$, as shown by Pönitz and Tittmann~\cite{Ponitz2000-yg}.

For all $d \geq 2$ we know that Maker wins the $(1,1)$ game on $\ZZ^d$ and that Breaker wins the (boosted) $(1, 2d-2)$ game on $\ZZ^d$, as this follows from \Cref{cor:deterministic-breaker}.
Hence, in higher dimensions, it is natural to ask the following question.

\begin{problem}
For each $d \geq 3$ determine the smallest $b$ such that Breaker wins the $(1,b)$ game on $\ZZ^d$.
\end{problem}

We also established with \Cref{thm:1dtrivial} that Breaker almost surely has a winning strategy for the $(1,d-1)$ game on $(\ZZ^d)_p$ with $p < 1$.
The behaviour for $p = 1$, however, is still undetermined.

\begin{problem}
Determine which of the players has a winning strategy for the $(1,d-1)$ game on $\ZZ^d$ when $d \geq 3$.
\end{problem}

Regardless of who wins the $(1,d-1)$ game on $\ZZ^d$, we take the fact that Breaker wins the $(1,d-1)$ game on $(\ZZ^d)_p$ for all $p < 1$ as a strong indication that Breaker wins the $(1,d)$ game on $\ZZ^d$.

\subsection{A heuristic argument for critical bias ratios}

Consider the $(m, \rho m)$ game on $\ZZ^2$ where $m$ is large.
If during optimal play, every edge gets eventually claimed, then a proportion $1/(1 + \rho)$ of the edges will be claimed by Maker.
Following the probabilistic intuition, as discussed in \Cref{subsec:phase-diagrams}, if both players play in a random-like manner then one can think of the graph spanned by the edges claimed by Maker as $(\ZZ^2)_p$, where $p = 1/(1+\rho)$.
Then, one could argue that as $\pcbond(\ZZ^2) = 1/2$, Maker should win if
\begin{equation*}
    \frac{1}{1 + \rho} > \pcbond(\ZZ^2) = \frac{1}{2},
\end{equation*}
or, equivalently, if $\rho < 1$.
Similarly, Breaker should win if $1/(1 + \rho) \leq 1/2$, so if $\rho \geq 1$.
We consider the following conjecture as our most important one.
\begin{conjecture}
The critical bias ratio of $\ZZ^2$ is $\rho(\ZZ^2) = 1$.
\end{conjecture}

For general lattices, the same heuristic leads to the following conjecture.

\begin{conjecture}
\label{conj:ratio}
If $\Lambda$ is a connected vertex-transitive graph, then
\begin{equation*}
    \rho(\Lambda) = \frac{1}{\pcbond(\Lambda)} - 1.
\end{equation*}
\end{conjecture}

If true, we expect this conjecture to be quite hard, especially for its generality.

If we were to conjecture what should be the full phase diagram for the $(1,1)$ game $(\ZZ^2)_{\alpha,\beta}$ we would suggest the phase diagram as in\Cref{fig:phaseconjecture}.
Compare with \Cref{fig:phase11} for the regions we have already determined.
While in principle the boundary between the red and blue regions could be an arbitrary monotone curve, it must touch the boundary of the triangle $\rT$ in the points $(0,0)$ and $(0.5,0.5)$.
While that may be enough evidence to believe that \Cref{fig:phaseconjecture} describes the correct partition, there are better reasons to do so.
If it is the case that $\rho(\ZZ^2) = 1$, than this amounts to say that Maker and Breaker are asymptotically equally strong in $\ZZ^2$.
This interpretation is appealing as $\ZZ^2$ is self-dual, and there are roughly as many self-avoiding walks from $v_0$ as there are dual cycles around $v_0$, of a given length, and these can be thought of as winning sets for Maker and Breaker, respectively.

\begin{figure}[ht!]
\centering
\begin{tikzpicture}[thick, scale=1]
\begin{axis}[
    width=10em,height=10em,
    scale only axis, axis equal,
    axis x line* = bottom,
    axis y line* = left,
    axis line style={draw=none},
    ymin=-0.04, ymax=1.04,
    xmin=-0.04, xmax=1.04,
    xtick = {0, 1/2, 1},
    xticklabels = {0, 1/2, 1},
    ytick = {0, 1/2, 1},
    yticklabels = {0, 1/2, 1},
    xlabel=$\alpha$,
    ylabel=$\;\;\beta$,
    every axis x label/.style={at={(current axis.right of origin)},anchor=west},
    every axis y label/.style={at={(current axis.north west)},anchor=south},
]

\filldraw[ultra thick, draw=blue, fill=blue!20] (0,0) -- (1,0) -- (0.5,0.5);
\filldraw[ultra thick, draw=red, fill=red!20] (0,1) -- (0.5,0.5) -- (0,0) -- cycle;
\filldraw[blue] (0,0) circle (0.2em);
\filldraw[red] (0.5,0.5) circle (0.2em);
\end{axis}
\end{tikzpicture}
\caption{Tentative phase diagram for the $(1,1)$ game on $(\ZZ^2)_{\alpha,\beta}$.}
\label{fig:phaseconjecture}
\end{figure}
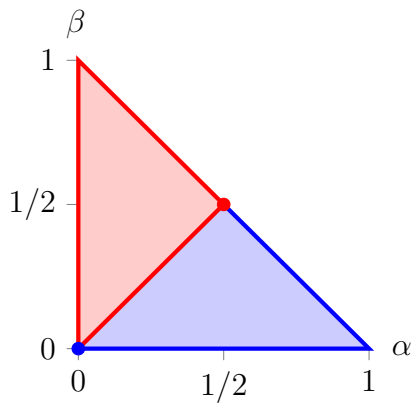

The diagram in \Cref{fig:phaseconjecture} does not only imply that the boundary between the two regions is the curve $\alpha = \beta$, but also that Breaker wins the $(1,1)$ game on the critical line if $\alpha, \beta > 0$.
While it may be quite ambitious to conjecture the behaviour on the phase boundary, we believe that overall Breaker has a slight advantage over Maker.
The pairing strategy employed by Maker when $\alpha = \beta = 0$ makes the colouring very far from random and breaks the heuristic argument.
If there is some true randomness in the initial configuration, then we believe that Breaker has the advantage.
We must emphasise that in principle, even if the boundary between the phases is indeed the line $\alpha = \beta$, then any colouring of the line is possible, as no monotonicity relationship holds between the points on this line.

\begin{problem}
Is the phase diagram for the $(1,1)$ Maker-Breaker percolation game on $(\ZZ^2)_{\alpha,\beta}$ as in \Cref{fig:phaseconjecture}?   
\end{problem}

In fact, we deem the determination of the phase diagram for the $(1,1)$ game on $(\ZZ^2)_{\alpha, \beta}$ to be of central importance.
Of course, it is also of great interest to fully determine the phase diagram for the $(m,b)$ game on $(\ZZ^2)_{\alpha,\beta}$ when $(m,b) \neq (1,1)$, but that may prove even harder.
In this direction, we single out the two compelling questions below.

We have seen from \Cref{thm:21nontrivial} that $\pcmb{(2,1)}(\ZZ^2)$ is non-trivial, meaning it lies in the open interval $(1/2,1)$.
In fact, $(\alpha,\beta) = (0,1-\pcmb{(2,1)}(\ZZ^2))$ is one of the points in the phase diagram where the phase boundary touches the boundary of the triangle $\rT$, the other being the point $(\alpha,\beta) = (1/2,1/2)$.

\begin{problem}
Determine the value of $\pcmb{(2,1)}(\ZZ^2)$, or equivalently, the value of
\begin{equation*}
    \beta^\ast \defined \inf \set[\big]{0 \leq \beta \leq 1 \st \text{almost surely, Breaker wins the $(2,1)$ game on $(\ZZ^2)_{0,\beta}$}}.
\end{equation*}
\end{problem}

We also pose the analogous question for the $(1,2)$ game.

\begin{problem}
Determine the value of 
\begin{equation*}
    \alpha^\ast \defined \inf \set[\big]{0 \leq \alpha \leq 1 \st \text{almost surely, Maker wins the $(1,2)$ game on $(\ZZ^2)_{\alpha,0}$}}.
\end{equation*}
\end{problem}

We know that $0.52784 \leq \pcmb{(2,1)}(\ZZ^2) \leq 0.94013$ from \Cref{thm:21nontrivial}, so we have the bounds $0.05987 \leq \beta^\ast \leq 0.47216$.
The bounds $0.05987 \leq \alpha^\ast \leq 1/2$ follow from \Cref{thm:potential-breaker-symmetric}.
The lower bounds we obtained for $\alpha^\ast$ and $\beta^\ast$ are both equal to ${(3/\kappa(\ZZ^2) - 1)/2}$.
Recall \Cref{fig:phase2112} for the current known regions in the phase diagrams of the $(2,1)$ and the $(1,2)$ games on $(\ZZ^2)_{\alpha,\beta}$.

Similarly to thresholds in bond percolation, there is no special reason to believe that these constants have nice closed form expressions.
On the other hand, even estimating them numerically could prove to be very challenging.

\subsection{Hexagonal and triangular lattices}

Denote by $\HH$ and $\TT$ the hexagonal and triangular lattices, respectively.
Recall that Duminil-Copin and Smirnov~\cite{Duminil-Copin2012-jj} determined that $\kappa(\HH) = \sqrt{2 + \sqrt{2}}$.
As for the triangular lattice, the exact value of $\kappa(\TT)$ is unknown, however, it was shown by Alm~\cite{Alm1993-tn} that $\kappa(\TT) \leq 4.278$.
With these bounds in hand, we determined in \Cref{cor:deterministic-breaker} that Breaker wins the $(1,1)$ and the $(2,3)$ game on $\HH$.
This begs the following question.

\begin{problem}
Determine which of the players has a winning strategy for the $(2,1)$ and the $(2,2)$ games on $\HH$.
\end{problem}

As for the triangular lattice, \Cref{cor:deterministic-breaker} gives us that Breaker wins the  $(1,4)$ and the $(2,18)$ game on $\TT$.
We now describe a simple pairing strategy\footnote{In the language of Section 3 of Day and Falgas-Ravry~\cite{Day2021-ua}, these strategies can be obtained as a consequence of the fact that the triangular lattice is \emph{$3$-path-colourable}.} to show that Maker wins the $(1,1)$ and the $(2,2)$ games on $\TT$.

In the standard embedding triangular lattice into $\RR^2 \cong \CC$, every vertex $x$ is connected to a vertex $x + \omega^{k}$, where $\omega \defined e^{2 \pi i / 6}$ is a root of unity and $0 \leq k \leq 5$.
Partition the edges of $\TT$ into triples of edges of the form
\begin{equation*}
    T_x \defined \set[\big]{ \set{x, x + 1}, \set{x, x + \omega}, \set{x,x + \omega^{-1}}}.
\end{equation*}
In other words, $T_x$ consists of the edges from $x$ whose endpoint is on the east of $x$.

To win the $(1,1)$ game on $\TT$, Maker reacts to Breaker in the following way.
If Breaker claims an edge in $T_x$, then Maker claims another edge in $T_x$, if possible.
The only occasion where this is not possible is when Breaker claims an edge in $T_x$ for the second time, in which case Maker already has an edge in $T_x$.
In any case, we maintain the invariant that no set $T_x$ can be fully claimed by Breaker.
This implies that there is always an infinite path from the origin with no edges of Breaker, hence Maker wins.

For the $(2,2)$ game, if Breaker claims two edges in a single $T_x$ in a given turn, Maker claims the missing edge from $T_x$.
Otherwise, if Breaker claims two edges, one in $T_x$ and another in $T_y$, Maker responds by claiming one edge in $T_x$ and another in $T_y$, if possible.
The effect is the same and no $T_x$ can be fully claimed by Breaker.

\begin{problem}
Determine which of the players has a winning strategy for the $(1,2)$ and the $(1,3)$ games on $\TT$.
\end{problem}

Another interesting question is to determine the smallest $b$ such that Breaker wins the $(2,b)$ game on $\TT$.
With the aforementioned strategy and \Cref{cor:deterministic-breaker}, we know that $3 \leq b \leq 18$.

The bond percolation thresholds for the triangular and hexagonal lattices have been determined by Sykes and Essam~\cite{Sykes1964-xs}, who proved that $\pcbond(\HH) = 1 - 2 \sin(\pi/18) \approx 0.6527$ and $\pcbond(\TT) = 2 \sin(\pi/18) \approx 0.3473$.
Thus, if \Cref{conj:ratio} is true, that would give
\begin{align*}
    \rho(\HH) &= \frac{1}{1 - 2 \sin(\pi/18)} - 1 \approx 0.532, \\
    \rho(\TT) &= \frac{1}{2 \sin(\pi/18)} - 1 \approx 1.879.
\end{align*}
Even though the critical ratio $\rho$ is an asymptotic quantity, this could suggest that Breaker wins the $(2,1)$ game on $\HH$ and that Maker wins the $(1,2)$ game on $\TT$.
See Wallwork~\cite{Wallwork2022-kb} for some results for Maker-Breaker crossing games on the triangular lattice, which can be seen as a finitary version of the Maker-Breaker percolation game, also introduced by Day and Falgas-Ravry~\cite{Day2021-sa}.


\section*{Acknowledgements}

The authors would like to thank their PhD supervisor Professor Béla Bollobás for his continued support and his helpful comments.
We would also like to thank Victor Falgas-Ravry for fruitful conversations on this subject.


\bibliographystyle{amsplain}



\end{document}